\pdfoutput=1
\documentclass[11pt]{article}
\usepackage[margin=3cm]{geometry}
\usepackage[utf8]{inputenc}
\usepackage{amsmath}
\usepackage{amsfonts}
\usepackage{amssymb}
\usepackage{amsthm}
\usepackage{mathrsfs}
\usepackage{commath}
\usepackage{latexsym}
\usepackage{accents}
\usepackage{array}
\usepackage{bm}
\usepackage{enumerate}
\usepackage[mathscr]{eucal}
\usepackage{color}
\usepackage{verbatim}
\usepackage{graphicx}
\usepackage{booktabs,caption}
\usepackage{multicol,multirow,tabularx}
\usepackage{siunitx}
\usepackage{relsize}
\usepackage{multirow}
\usepackage{hyperref}
\usepackage{float}
\usepackage{cite}
\usepackage{todonotes}
\usepackage{tikz}
\newcommand*{\doi}[1]{\href{http://dx.doi.org/#1}{doi: #1}}
\usepackage[square,sort,comma,numbers,compress]{natbib}

\newtheorem{theorem}{Theorem}
\newtheorem{lemma}[theorem]{Lemma}
\newtheorem{proposition}[theorem]{Proposition}

\usetikzlibrary{arrows.meta}
\usetikzlibrary{positioning}
\usetikzlibrary{decorations.pathreplacing,decorations.markings}
\tikzset{->-/.style={decoration={
markings,
mark=at position #1 with {\arrow{>}}},postaction={decorate}}}
\usepackage{pgfplots}
\pgfplotsset{width=7cm,compat=newest,ticks=none}
\usepgfplotslibrary{fillbetween}

\definecolor{royalpurple}{rgb}{0.47, 0.32, 0.66}
\definecolor{pastelgreen}{rgb}{0.47, 0.87, 0.47}
\definecolor{cornellred}{rgb}{0.7, 0.11, 0.11}
\definecolor{pastelorange}{rgb}{1.0, 0.7, 0.28}
\definecolor{darkred}{rgb}{0.55, 0.0, 0.0}
\definecolor{darkpastelgreen}{rgb}{0.01, 0.75, 0.24}

\title{Far-from-equilibrium travelling pulses in sloped semi-arid environments driven by autotoxicity effects}
\author{Gabriele Grif\`o$^{1,2}$, Annalisa Iuorio$^{3,4,\ast}$, 
Frits Veerman$^{5}$\\ 
{$^{1}$\small Department of Mathematical, Computer, Physical and Earth Sciences, University of Messina,} \\
{\small V.le F. Stagno D'Alcontres 31, I-98166 Messina, Italy.}\\
{$^{2}$\small Department of Engineering, University of Messina, C.da di Dio, I-98166 Messina, Italy}\\
{$^{3}$\small Department of Engineering, Parthenope University of Naples,}\\ {\small Centro Direzionale - Isola C4, 80143 Naples, Italy}\\
{$^{4}$\small Faculty of Mathematics, University of Vienna, Oskar-Morgenstern-Platz 1, 1090 Vienna, Austria}\\
{$^{5}$\small Mathematical Institute, University of Leiden, Einsteinweg 55, 2333 CC Leiden, The Netherlands}\\
{\small $^{(\ast)}$ Corresponding author: annalisa.iuorio@uniparthenope.it}}
\date{}

\begin{document}

\maketitle

\begin{abstract}
In this work, an extension of the 1D Klausmeier model that accounts for the toxicity compounds is considered and the occurrence of travelling stripes is investigated. Numerical simulations are firstly conducted to capture the qualitative behaviours of the pulse--type solutions and, then, geometric singular perturbation theory is used to prove the existence of such travelling pulses by constructing the corresponding homoclinic orbits in the associated $4$-dimensional system. A scaling analysis on the investigated model is performed to identify the asymptotic scaling regime in which travelling pulses can be constructed. Biological observations are extracted from the analytical results and the role of autotoxicity in travelling patterns is emphasized. Finally, the analytically constructed solutions are compared with the numerical ones, leading to a good agreement that confirms the validity of the conducted analysis. Numerical investigations are also carried out in order to gain additional information on vegetation dynamics.\\

\end{abstract}

\section{Introduction}

One of the major challenges in the context of climate change is desertification, which poses a danger for the durability of dryland areas \cite{UN22}. 
To predict and mitigate desertification catastrophes, considerable effort has been spent to study ecosystem response to rapid change, and to identify indicators that provide reliable information on ecosystem health and robustness \cite{de1998early,Tirabassi2014}. In this context, the occurrence and dynamics of vegetation patterns have been increasingly recognised as a key predictive indicator \cite{Gowda2018,May1977,Rietkerk1997site,Saco2018}. Recent studies \cite{Bastiaansen2020, Rietkerk2021} show how vegetation patterns can be interpreted as a signature of ecosystem resilience, providing a gradual (in contrast to catastrophic) sequence of patterned states towards the desert one. 


In arid and semi-arid environments, the self-organisation of biomass can mainly be attributed to local positive feedback between water and biomass \cite{Meron2015}. However, vegetation patterns are also observed in non-arid (non water-limited) ecosystems, suggesting the presence of alternative mechanisms that may contribute to the emergence of spatial vegetation patterns. One of these mechanisms, called \emph{autotoxicity}, is an example of plant-soil negative feedback that arises from the presence of soil-borne pathogens, changing the composition of soil microbial communities, and/or the accumulation of autotoxic compounds from decomposing plant litter \cite{Mazzoleni2007, Mazzoleni2015}.


Although the emergence of stationary and oscillatory vegetation patterns is believed to be ubiquitous, these patterning phenomena tend to occur over long timescales and far away from populated areas, making on site experiments and detailed observations challenging to perform. Therefore, mathematical modelling plays a crucial role in analysing these pattern formation processes, providing analytical tools to assess an ecosystem's response to stressors.
There is an extensive body of mathematical literature on reproducing spatially periodic vegetation patterns and predicting their spatial evolution, using various models \cite{Klausmeier1999, Hillerislambers2001, Rietkerk2002, Sherratt2005, Siteur2014, Meron2015, Carteni2012, VanDerStelt2013, Marasco2020, Eigentler2020, Iuorio2023pre, Iuorio2023preII, Spiliotis2023sub}.
Among these models, the Klausmeier model \cite{Klausmeier1999} is one of the most simple systems still able to reproduce the occurrence of vegetation patterns as migrating bands. The Klausmeier model is a two-compartment reaction--advection--diffusion model describing water and biomass evolution in arid and semi-arid enviroments. However, this model does not take other processes into account that may influence vegetation dynamics. Therefore, several extensions to the Klausmeier model have been proposed, that include: 
(i) the emergence of stationary patterns on flat terrains \cite{Kealy2012, Zelnik2013, Sun2013, Consolo2022II, Grifo2023II, Curro2023}, (ii) the influence of inertial effects \cite{Consolo2022PRE, Consolo2022II, Consolo2022III, Grifo2023, Grifo2023II, Curro2023}, (iii) the presence of secondary seed dispersal \cite{Thompson2009, Thompson2014, Consolo2022III, Grifo2023}, (iv) the occurrence of toxicity compounds \cite{Marasco2013, Marasco2014, Iuorio2021, Consolo2024}, and (v) a finite soil carrying capacity \cite{Bastiaansen2019, Eigentler2021, Consolo2023, Byrnes2023}.\\

To the best of our knowledge, the existing literature on water-biomass-toxicity systems focuses on flat semi-arid environments \cite{Marasco2013, Marasco2014, Iuorio2021, Consolo2024}. We aim to eludicate the influence of autotoxicity on migrating vegetation bands in sloped semi-arid terrains, by means of a generalisation of the Klausmeier model. We study far--from--equilibrium solutions using geometric singular perturbation theory, in particular solutions that are bi-asymptotic to the desert state. We also present the results of numerical simulations that reflect our analytical results, and that allow us to extract additional information on the influence of autotoxicity on travelling vegetation pulses.

This paper is organised as follows. In Section \ref{sec:model}, we introduce an extension of the original Klausmeier model, adding a toxicity component. In the same section, we present the results of numerical simulations that show the emergence of migrating vegetation bands. In Section \ref{sec:Existence}, a scaling analysis is performed to identify the asymptotic scaling regime in which travelling vegetation bands can be found. We construct a pulse--type solution in the singular limit, and prove its persistence. In Section \ref{sec:numcont}, we use numerical continuation methods to study the behaviour of the constructed solution for a variety of parameter ranges, to illustrate and extend the asymptotic analysis in Section \ref{sec:Existence}. Future directions and concluding remarks are presented in Section \ref{sec:Conclusion}.

\section{The model}
\label{sec:model}
Following existing literature in which the influence of toxic compounds on vegetation dynamics is studied \cite{Carteni2012, Marasco2013, Marasco2014, Iuorio2021}, we formulate a generalisation of the two-compartment Klausmeier model \cite{Klausmeier1999}, incorporating the effect of autotoxicity. In particular, we add a ordinary differential equation (ODE) that describes the temporal evolution of a toxic compound. Additional kinetic terms describe the interaction between the biomass and the toxic compound. Moreover, assuming a ramp-like topography, we reduce the spatial dimension of the model to focus on the principal direction of the slope. The resulting model reads


\begin{equation}
\begin{aligned}
W_t &= p - r B^2 W - l W + \nu W_x, \medskip \\
B_t &= c B^2 W - (d + sT) B + D_B B_{xx}, \medskip \\
T_t &= q (d + sT) B - (k + wp)T,
\end{aligned}
\label{eq:dimmod}
\end{equation}
in which $W(x,t)$, $B(x,t)$, and $T(x,t)$ represent the surface water, biomass, and autotoxicity densities at location $x \in \Omega \subset{\mathbb{R}}$ (the positive $x$-direction being uphill) and time $t \in \mathbb{R}^+$, respectively. 

The motion of surface water $W$ is dominated by anisotropic transport, modelled by an advection term with speed $\nu>0$, mimicking downhill water flow. The biomass $B$ is assumed to diffuse isotropically with diffusion coefficient $D_B$. In addition, the biomass density increases due to water availability with growth rate $c$ and decreases both due to intrinsic mortality with rate $d$, and due to the presence of toxic compounds with sensitivity $s$. The surface water density is fed by precipitation $p$ (assumed to be constant over the intrinsic time scale of our model), and decreases both due to evaporation or drainage with rate $l$ and due to water uptake by biomass with rate $r$. The additional toxicity component $T$ decreases due to natural decay with rate $k$, and is washed out by precipitation proportional to $w$. The toxicity increases due to the decomposition of biomass, proportional to $q$. For a detailed description of all the parameters of system \eqref{eq:dimmod}, we refer to Table \ref{table1}.\\
\begin{table}[b!]
\centering
\begin{tabular}{l l l}
    \hline
    Parameter & Description & Units \\
    \hline
    $c$ & Growth rate of $B$ due to water uptake & m$^4$d$^{-1}$kg$^{-2}$ \\
    $d$ & Death rate of biomass $B$ & d$^{-1}$ \\
    $k$ & Decay rate of toxicity $T$ & d$^{-1}$ \\
    $l$ & Water loss due to evaporation or drainage & d$^{-1}$ \\
    $p$ & Precipitation rate & kg d$^{-1}$ m$^{-2}$ \\
    $q$ & Proportion of toxins in dead biomass & - \\
    $r$ & Rate of water uptake & m$^4$d$^{-1}$kg$^{-2}$ \\
    $s$ & Sensitivity of plants to toxicity $T$ & m$^2$d$^{-1}$kg$^{-1}$ \\
    $w$ & Washing out of toxins by precipitation & m$^2$kg$^{-1}$ \\
    $D_B$ & Diffusion coefficient of biomass $B$ & m$^2$d$^{-1}$ \\
    $\nu$ & Water advection speed & m\, d$^{-1}$ \\
    \hline
\end{tabular}
\\
\caption{An overview of the parameters of model \eqref{eq:dimmod}. 
}
\label{table1}
\end{table}

We nondimensionalise model \eqref{eq:dimmod} by applying the rescaling

\begin{equation}\label{eq:rescaling_nondimensionalisation}
\tilde{x}=\sqrt{\frac{l}{D_B}} x,
\hspace{0.5cm}
\tilde{t}=l\, t, 
\hspace{0.5cm}
W=\frac{\sqrt{l r}}{c}U,
\hspace{0.5cm}
B=\sqrt{\frac{l}{r}}V, 
\hspace{0.5cm}
T=\frac{q l \sqrt{l}}{\left( k + pw \right)\sqrt{r}} S.     
\end{equation}

The associated nondimensional parameters are defined as

\begin{equation}\label{eq:rescaling_nondimensionalisation_pars}
\begin{array}{c}
\mathcal{A} = \frac{c p}{l \sqrt{l r}},
\hspace{0.7cm}
\mathcal{B} = \frac{d}{l},
\hspace{0.7cm}
\mathcal{D} = \frac{l}{k + p w}, 
\medskip \\
\mathcal{H} = \frac{s q \sqrt{l}}{(k + p w)\sqrt{r}},
\hspace{0.7cm}
\varepsilon=\frac{\sqrt{D_B l}}{\nu}. 
\end{array}
\end{equation}

Here, $\mathcal{A}$ is proportional to the precipitation rate $p$ and hence can be interpreted (by fixing all other dimensional parameters) as its nondimensional counterpart. The parameter $\mathcal{B}$ (resp.~$\mathcal{D}$) represents the ratio between the (linear) biomass and water (resp.~water and toxicity) loss. 
Analogous to $\mathcal{A}$, the parameter $\mathcal{H}$ is directly proportional to the biomass sensitivity to autotoxicity $s$; therefore, it can be interpreted (considering all other parameters fixed) as the nondimensional version of this parameter. The parameter $\varepsilon$ quantifies the ratio between the (normalised) diffusion speed of the biomass and the downhill advection speed of the surface water. As the downhill flow of water is usually significantly faster than the spread of biomass due to growth, we assume that $\varepsilon$ is asymptotically small, i.e. $0<\mathcal{\varepsilon}\ll 1$. This parameter will play a pivotal role in the analysis in Section \ref{sec:Existence}.\\
Using the rescaling \eqref{eq:rescaling_nondimensionalisation} and the nondimensional parameters \eqref{eq:rescaling_nondimensionalisation_pars}, system \eqref{eq:dimmod} becomes (dropping the tilde notation for the dependent variables $\tilde{x}$ and $\tilde{t}$)

\begin{equation}
\begin{aligned}
U_{t} &= \mathcal{A} - U - V^2 U +\varepsilon^{-1} U_{x}, \medskip \\
V_{t} &= V^2 U - \mathcal{B} V - \mathcal{H} S V + V_{xx}, \medskip \\
\mathcal{D} S_{t} &= \mathcal{B} V + \mathcal{H} S V - S.
\end{aligned}
\label{eq:ndimmod}
\end{equation}
Note that the original Klausmeier model \cite{Klausmeier1999} is a special case of model \eqref{eq:ndimmod}: taking $\mathcal{H}=0$ fully decouples the autotoxicity equation.\\

\subsection{Preliminary observations and main result}
The influence of autotoxicity can already be observed in the spatially homogeneous steady states of system \eqref{eq:ndimmod}. In addition to the trivial desert state
\begin{equation}
\label{eq:equil0}
\left(U_\ast,V_\ast,S_\ast\right)=\left(\mathcal{A},0,0\right),
\end{equation}
system \eqref{eq:ndimmod} admits two nontrivial vegetated states
\begin{equation}
\label{eq:equilpm}
\left(U_{\pm},V_{\pm},S_{\pm}\right)=\left(\frac{\mathcal{A}}{1 + V_\pm^2},\frac{\mathcal{A}\pm \sqrt{\mathcal{A}^2-4\mathcal{B}\left(\mathcal{B}+\mathcal{A}\mathcal{H}\right)}}{2\left(\mathcal{B}+\mathcal{A}\mathcal{H}\right)},\frac{\mathcal{B}V_{\pm}}{1-\mathcal{H}V_{\pm}}\right)
\end{equation}
which explicitly depend on the value of the toxicity coupling parameter $\mathcal{H}$. In particular, these nontrivial spatially homogeneous steady states exist if and only if $\mathcal{A}>2\mathcal{B}\left(\mathcal{H}+\sqrt{1+\mathcal{H}^2}\right)$, which provides a nontrivial condition on the original `Klausmeier' parameters $\mathcal{A}$ and $\mathcal{B}$. This condition reduces to the non-toxic Klausmeier condition $\mathcal{A} > 2 \mathcal{B}$ in the limit $\mathcal{H}=0$. 
\\

Existing theoretical and experimental works \cite{Tongway2001, Dunkerley2018,Carter2018,Sewalt2017} show that vegetation patterns on sloped semi-arid terrain are predominantly characterised by pulse or wave train solutions that originate from the desert state. From a mathematical viewpoint, these patterns are determined by far--from--equilibrium dynamics that are beyond the reach of classical tools such as linear or weakly nonlinear stability analysis \cite{Murray,Hoyle,Doelman2018in}.\\ Direct numerical simulation of model \eqref{eq:ndimmod} confirms this observation. Using Matlab\textsuperscript{\textregistered} \cite{MATLAB}, we integrate \eqref{eq:ndimmod} over a spatial domain of length $L=1000$ with periodic boundary conditions, with constant initial data for both water $U$ and autotoxicity $S$, i.e. $U(x,0) = 0.5$ and $S(x,0) = 0$ $\forall x \in [0,L]$. For the vegetation variable $V$, we take as initial data a Gaussian pulse centered in $x=300$ with standard deviation $\sigma=\frac{2}{5}$ and amplitude $V_{max}=10$. The model parameters are fixed (in accordance with \cite{Klausmeier1999, Carter2018, Iuorio2021}) as  $\mathcal{A}=1.2$, $\mathcal{B}=0.45$, $\varepsilon=0.005$, $\mathcal{D}=4.5$ and $\mathcal{H}=1$. In Figure \ref{fig:num}(d), the space--time evolution of the biomass density clearly indicates a narrow peak travelling at constant speed through the domain. Figure \ref{fig:num} a)--c) depicts the spatial profile of this travelling pulse, highlighting the nonlinear and spatially multi-scale nature of this solution. In particular, the component profiles suggest a hierarchy of three asymptotically distinct spatial scales, which we denote as `superslow', `slow' and `fast'. Note that this terminology is slightly misleading, as these scales are not connected to the \emph{temporal} evolution of the pulse profile; rather, they indicate how the pulse profile changes as a function of $x$ (for fixed $t$). Using these terms, $V$ is observed to change on the fast spatial scale; the profile of $S$ displays both slow and fast changes, and the profile of $U$ encorporates both superslow and fast behaviour. This spatial scale separation will play a crucial role in the analysis in Section \ref{sec:Existence}.\\

\begin{figure}[h!]
	\centering
	\includegraphics[width=1\textwidth]{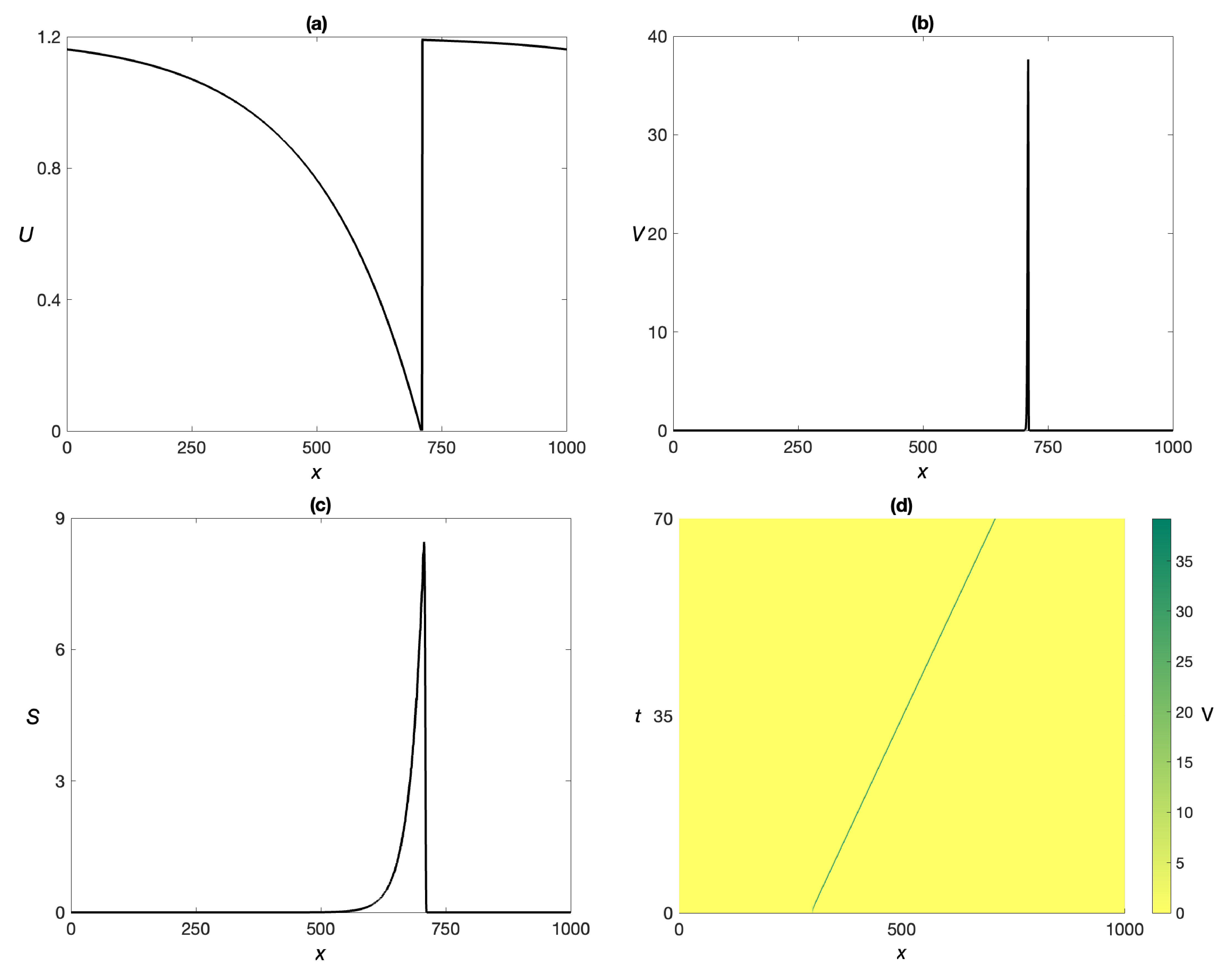}
	\caption{
 Spatial profiles for surface water $U$ (a), biomass $V$ (b) and toxicity $S$ (c) of a travelling pulse, obtained by numerical integration of system \eqref{eq:ndimmod} over a domain with length $L=1000$ with periodic boundary conditions. The parameter values are $\mathcal{A}=1.2$, $\mathcal{B}=0.45$, $\varepsilon=0.005$, $\mathcal{D}=4.5$, and $\mathcal{H}=1$.\\
}
	\label{fig:num}
\end{figure}

Inspired by the observations depicted in Figure \ref{fig:num}, we investigate in this paper the existence of travelling pulses with a fixed wave speed $\mathcal{C}>0$. The existence of far--from--equilibrium pulse--type solutions has been studied in previous works, both in the context of specific models \cite{Carter2018, Bastiaansen2019,Veerman2013} and general classes of reaction-diffusion systems \cite{Doelman2015,Doelman2001}, using geometric singular perturbation theory (GSPT). We follow an analogous approach for our model \eqref{eq:ndimmod} that includes autotoxicity. As such, our work can be compared directly to \cite{Carter2018}, where GSPT was used to establish the existence of travelling pulses in \eqref{eq:ndimmod} in the absence of autotoxicity, i.e. for $\mathcal{H}=0$. In comparison with \cite{Carter2018}, we show that autotoxicity not only introduces novel behaviour --which is, from a general viewpoint, to be expected when considering a model of higher complexity-- but also \emph{simplifies} the asymptotic analysis by avoiding the need to use geometric blow-up techniques, as was necessary in \cite{Carter2018}.\\

Our main analytical result is presented in the following Theorem:
\begin{theorem} \label{thm:main}
For $0 < \varepsilon \ll 1$ sufficiently small, there exists a unique $\theta_0>0$ such that there exists a travelling wave solution $\left( U, V, S \right) (x,t) = \left( U, V, S \right) (x-\mathcal{C} t)$ of system \eqref{eq:ndimmod} with wave speed $\mathcal{C}=\left( \frac{\mathcal{A}^2 \theta_0^2}{\varepsilon} \right)^{1/3} + \mathcal{O}(1)$.
\end{theorem}

The constant $\theta_0$ coincides with the one provided in \cite[Theorem 1.1]{Carter2018}, i.e.~\mbox{$\theta_0 \approx 0.8615$}.

\newpage

\section{Existence of travelling pulses} \label{sec:Existence}
In this section, we apply GSPT to prove the existence of a constant speed travelling pulse solution to \eqref{eq:ndimmod}. 
Introducing a comoving frame variable $\xi = x - \mathcal{C} t$, system \eqref{eq:ndimmod} can be rephrased as
\begin{equation}
\begin{aligned}
\left( \mathcal{C} + \frac{1}{\varepsilon} \right) U_\xi + \mathcal{A} - U - U V^2 &= 0, \medskip \\
V_{\xi \xi} + \mathcal{C} V_\xi - \mathcal{B} V + U V^2 - \mathcal{H} S V &= 0, \medskip \\
\mathcal{C} \mathcal{D} S_\xi + \mathcal{B} V + \mathcal{H} S V - S &= 0.
\end{aligned}
\label{travel}
\end{equation}
System \eqref{travel} is equivalent to the first-order system
\begin{equation}
\begin{aligned}
U_\xi &= \frac{\varepsilon}{1 + \varepsilon \mathcal{C}} \left(U - \mathcal{A} + U V^2 \right), \medskip \\
V_\xi  &= Q, \medskip \\
Q_\xi &= \mathcal{B} V - U V^2 + \mathcal{H} S V - \mathcal{C} Q, \medskip \\
\mathcal{C} \mathcal{D} S_\xi &= S - \mathcal{B} V - \mathcal{H} S V. 
\end{aligned}
\label{eq:prescaling}
\end{equation}
Inspired by \cite{Carter2018} and by the different spatial scales visible in the pulse profile shown in Figure \ref{fig:num}, we introduce the following asymptotic rescaling:
\begin{equation}
\begin{array}{c c c}
V = c^{-1} \varepsilon^{-2/3} v, &
Q = \varepsilon^{-1} q, &
\xi = c^2 \varepsilon^{1/3} \tau, \medskip \\
\mathcal{C} = \varepsilon^{-1/3} c, &
S = c^{-1} \varepsilon^{-2/3} s, &
\delta = \varepsilon^{2/3} c.
\end{array}
\label{eq:rescaling}
\end{equation}
In these rescaled variables, system \eqref{eq:prescaling} becomes
\begin{equation}
\begin{aligned}
\dot{u} &= \frac{1}{1+\delta} \left( u v^2 + \delta^2 \left( u - \mathcal{A} \right) \right), \medskip \\
\dot{v} &= c^3 q, \medskip \\
\dot{q} &= \delta \mathcal{B} v - u v^2 + \mathcal{H} s v - c^3 q, \medskip \\
\mathcal{D} \dot{s} &= \delta s - \delta \mathcal{B} v - \mathcal{H} s v,
\end{aligned}
\label{eq:postscaling}
\end{equation}
where $\dot{} = \frac{\text{d}}{\text{d} \tau}$. Note that $\delta$ has taken over the role of $\varepsilon$ as the asymptotically small parameter. Single stripe patterns travelling uphill are given by orbits of \eqref{eq:postscaling} that are homoclinic to the equilibrium $\left(u, v, q, s\right)=\left( 
\mathcal{A}, 0, 0, 0 \right)$, which corresponds to the spatially homogeneous desert state $\left( U_\ast, V_\ast, S_\ast \right)$ in \eqref{eq:equil0}.\\
A priori, it is not immediately clear that system \eqref{eq:postscaling} admits \emph{three} (rather than two) asymptotic scales. Introducing the new variable
\begin{equation}\label{eq:definition_w_variable}
w = \left( 1 + \delta \right) u + q + v + \mathcal{D} s
\end{equation}
will aid the separation of scales in the upcoming analysis. Using $w$ in favour of $u$, system \eqref{eq:postscaling} is rephrased as
\begin{equation}
\begin{aligned}
\dot{w} &= \delta s + \frac{\delta^2}{1+\delta} \left( w - q - v - \mathcal{D} s - a \right), \medskip \\
\dot{v} &= c^3 q, \medskip \\
\dot{q} &= \delta \mathcal{B} v - \frac{1}{1 + \delta} \left( w - q - v - \mathcal{D} s \right) v^2 + \mathcal{H} s v - c^3 q, \medskip \\
\mathcal{D} \dot{s} &= \delta s - \delta \mathcal{B} v - \mathcal{H} s v,
\end{aligned}
\label{eq:postscalingw}
\end{equation}
where $a = \left( 1 + \delta \right) \mathcal{A}$. 
In the following, we will refer to system \eqref{eq:postscalingw} as the \emph{fast} system. Our goal is to construct a homoclinic orbit to the equilibrium $\left( w, v, q, s \right) = \left( a, 0, 0, 0 \right)$.\\
The fast-slow structure of system \eqref{eq:postscalingw} induced by the presence of the small parameter \mbox{$0 < \delta \ll 1$} allows us to construct the solution described in Theorem \ref{thm:main} as perturbations of singular orbits obtained by matching portions of critical manifolds (i.e.~manifolds of equilibria of system \eqref{eq:postscalingw} for $\delta=0$) with fast jumps along heteroclinic orbits arising in the corresponding layer problem. The reformulation of system \eqref{eq:postscalingw} with respect to the slow and the superslow scale is described in Section \ref{sec:redprob}, whereas the construction of singular homoclinic orbits is illustrated in Section \ref{sec:singorb}. The persistence of singular orbits for $0 < \delta \ll 1$ stated in Theorem \ref{thm:main} is proved in Section \ref{sec:persist}. Our theoretical findings are illustrated and extended by numerical results obtained via the continuation software AUTO \cite{Doedel1981} in Section \ref{sec:numcont}.


\subsection{Critical manifolds and reduced dynamics} 
\label{sec:redprob}
We take $\delta=0$ in the fast system \eqref{eq:postscalingw} to obtain the layer problem
\begin{subequations} \label{eq:laypb}
\begin{align}
\dot{w} &= 0, \\
\dot{v} &= c^3 q, \label{laypbV} \\
\dot{q} &= - \left( w - q - v - \mathcal{D} s \right) v^2 + \mathcal{H} s v - c^3 q, \\
\mathcal{D} \dot{s} &= - \mathcal{H} s v. \label{laypbS}
\end{align}
\end{subequations}
The equilibria of the layer problem \eqref{eq:laypb} determine the critical manifolds
\begin{subequations}\label{eq:critmanifolds}
\begin{align}
\mathcal{M}^{(1)} &= \left\{\,v=0,\, q=0,\, s=0\,\right\},\\
\mathcal{M}^{(2)} &= \left\{\,v=w,\, q=0,\, s=0\,\right\},\\
\mathcal{M}^{(3)} &= \left\{\,v=0,\, q=0\,\right\}.
\end{align}
\end{subequations}
Geometrically, the one-dimensional critical manifolds $\mathcal{M}^{(1)}$ and $\mathcal{M}^{(2)}$ are lines, and $\mathcal{M}^{(1)}$ is contained in the two-dimensional hyperplane $\mathcal{M}^{(3)}$. For future reference, we denote the equilibria of the layer problem \eqref{eq:laypb} as
\begin{equation}
p_1(w) = (w, 0, 0, 0),
\hspace{0.3cm}
p_2(w) = (w, w, 0, 0),
\hspace{0.3cm}
p_3(w,s) = (w, 0, 0, s),
\label{eq:equilibria}
\end{equation}
allowing us to parametrise the critical manifolds as $\mathcal{M}^{(1)} = \bigcup_{w \in \mathbb{R}} p_1(w)$, $\mathcal{M}^{(2)} = \bigcup_{w \in \mathbb{R}} p_2(w)$ and $\mathcal{M}^{(3)} = \bigcup_{(w,s) \in \mathbb{R}^2} p_3(w,s)$. We note that $\mathcal{M}^{(1)}$ transversally intersects $\mathcal{M}^{(2)}$ (and therefore $\mathcal{M}^{(3)}$) at the origin. The critical manifold $\mathcal{M}^{(3)}$ is normally hyperbolic, since the equilibrium $p_3(w,s)$ is a saddle in the normal $(v,q)$-directions. The critical manifold $\mathcal{M}^{(2)}$ is normally hyperbolic away from the origin, with eigenvalues $-c^3$, $-\frac{\mathcal{H} w}{\mathcal{D}}$ and $w^2$ in the normal directions. Since $\mathcal{M}^{(1)}$ is fully contained in $\mathcal{M}^{(3)}$, its persistence is related to the dynamics on $\mathcal{M}^{(3)}$, which we will turn to next.\\

We observe that the hyperplane $\mathcal{M}^{(3)}$ \eqref{eq:critmanifolds} is invariant under the flow of \eqref{eq:postscalingw}. In addition, $\mathcal{M}^{(1)} \subset \mathcal{M}^{(3)}$ is itself also invariant under the flow of \eqref{eq:postscalingw}. In contrast to the other two critical manifolds, the critical manifold $\mathcal{M}^{(2)}$ is not invariant under the flow of the full system \eqref{eq:postscalingw}. The flow on $\mathcal{M}^{(2)}$ is to leading order given by
\begin{subequations}
\begin{align}
\dot{q} &= \delta \mathcal{B} w , \medskip \\
\mathcal{D} \dot{s} &= - \delta \mathcal{B} w.
\end{align}
\label{eq:slowM2}
\end{subequations}
Since $s=0$ on $\mathcal{M}^{(2)}$, the reduced dynamics here would be biologically unfeasible ($s$ decreases and hence becomes negative). The critical manifold $\mathcal{M}^{(2)}$ is therefore not going to play a role in the construction of the travelling pulse solutions (see Figure \ref{fig:critman}). Hence, we focus our attention on the dynamics on $\mathcal{M}^{(3)}$, which contain the dynamics on $\mathcal{M}^{(1)}$.\\

To study the reduced flow on the critical manifold $\mathcal{M}^{(3)}$, we introduce $\sigma = \delta \tau$. In this slow variable, system \eqref{eq:postscalingw} reads
\begin{subequations}
\begin{align}
w_\sigma &= s + \frac{\delta}{1+\delta} \left( w - q - v - \mathcal{D} s - a \right), \medskip \\
\delta v_\sigma &= c^3 q, \medskip \\
\delta q_\sigma &= \delta \mathcal{B} v - \frac{1}{1 + \delta} \left( w - q - v - \mathcal{D} s \right) v^2 + \mathcal{H} s v - c^3 q, \medskip \\
\delta \mathcal{D} s_\sigma &= \delta s - \delta \mathcal{B} v - \mathcal{H} s v.
\label{eq:slow}
\end{align}
\end{subequations}
The flow on $\mathcal{M}^{(3)}$ is therefore given by
\begin{subequations}
\begin{align}
w_\sigma &= s + \frac{\delta}{1+\delta} \left( w - \mathcal{D} s - a \right), \medskip \\
\mathcal{D} s_\sigma &= s.
\end{align}
\label{eq:slowflow_M3}
\end{subequations}
System \eqref{eq:slowflow_M3} is itself a (linear) slow-fast system, which can be brought in Fenichel normal form by observing that
\begin{equation}\label{eq:w-Ds_flow}
 \frac{\text{d}}{\text{d} \sigma}\left(w - \mathcal{D} s\right) = \frac{\delta}{1+\delta}\left((w-\mathcal{D} s)-a\right).
\end{equation}
First, we observe that the set $\left\{\,s=0\,\right\}$, representing $\mathcal{M}^{(1)}$, is invariant under the flow of  \eqref{eq:slowflow_M3}. At the same time, $\left\{\,s=0\,\right\}$ is the critical manifold of the slow-fast system \eqref{eq:slowflow_M3}; the flow on $\mathcal{M}^{(1)}$ is given by
\begin{equation}\label{eq:superslowM1}
 w_\zeta = \frac{1}{1+\delta} \left( w - a \right)
\end{equation}
in terms of the superslow variable $\zeta := \delta \sigma = \delta^2 \tau$. The flow \eqref{eq:superslowM1} has a single unstable equilibrium at $w=a$, which corresponds to the global desert state.\\
Second, the linear coordinate transform \eqref{eq:w-Ds_flow} inspires us to foliate the critical manifold $\mathcal{M}^{(3)}$ by one-dimensional manifolds 
\begin{equation}\label{eq:M3_w0}
\mathcal{M}^{(3),w_0}:= \left\{ p_3\left(w, \frac{w-w_0}{\mathcal{D}}\right) \, : \, w_0 \in \mathbb{R}_0^+ \right\},
\end{equation}
cf. \eqref{eq:equilibria}. Note that $\mathcal{M}^{(3),w_0}$ are precisely the (unstable) Fenichel fibres for the flow \eqref{eq:slowflow_M3}; the flow on $\mathcal{M}^{(3),w_0}$ is independent of $w_0$, and given by $\mathcal{D} s_\sigma = s$.
The critical manifolds and the reduced flow on these critical manifolds is shown in Figure \ref{fig:critman}.

\begin{figure}[H]
\begin{center}
\begin{tikzpicture}
\draw[thick,->] (0,0) -- (7,0);
\node [below right] at (7,0) {$v$};
\draw[thick,->] (0,0) -- (3.6,1.8);
\node [right] at (3.6,1.8) {$s$};
\draw[thick,->] (0,0) -- (0,6);
\node [left] at (-0.1,6) {$w$};

\draw[very thick,blue] (0,0) -- (6,6);
\begin{scope}[very thick,royalpurple,decoration={
    markings,
    mark=at position 0.5 with {\arrow{{Stealth}}}}
    ] 
    \draw[postaction={decorate}] (0,5.8) -- (0,0);
\end{scope}
\begin{scope}[very thick,royalpurple,decoration={
    markings,
    mark=at position 0.87 with {\arrow{{Stealth}}}}
    ] 
    \draw[postaction={decorate}] (0,0) -- (0,5.8);
\end{scope}
\draw[fill,royalpurple] (0,4) circle [radius=2pt];
\node[left] at (0,4) {$a$};
\node[left] at (0,2) {\textcolor{royalpurple}{$\mathcal{M}^{(1)}$}};
\node[right] at (6.3,6.1) {\textcolor{blue}{$\mathcal{M}^{(2)}$}};
\node[below left] at (1.2,5.9) {\textcolor{blue}{$\mathcal{M}^{(3)}$}};
\draw[royalpurple] (-.1,1.6) -- (-.4,1.7);
\draw[blue] (5.9,5.8) -- (6.3,5.9);

\fill[nearly transparent,blue] (0,0) -- (0,5.6) -- (3.4,7.4) -- (3.4,1.7);
\begin{scope}[very thick,nearly transparent,blue,decoration={
    markings,
    mark=at position 0.34 with {\arrow{{Stealth}{Stealth}}}}
    ] 
    \draw[postaction={decorate}] (0,3.9) -- (1.75,6.52);
\end{scope}
\begin{scope}[very thick,nearly transparent,blue,decoration={
    markings,
    mark=at position 0.37 with {\arrow{{Stealth}{Stealth}}}}
    ] 
    \draw[postaction={decorate}] (0,2.5) -- (3.2,7.28);
\end{scope}    
\begin{scope}[very thick,nearly transparent,blue,decoration={
    markings,
    mark=at position 0.53 with {\arrow{{Stealth}{Stealth}}}}
    ] 
    \draw[postaction={decorate}] (0,1.1) -- (3.38,6.17);
\end{scope}
\begin{scope}[very thick,nearly transparent,blue,decoration={
    markings,
    mark=at position 0.67 with {\arrow{{Stealth}{Stealth}}}}
    ] 
    \draw[postaction={decorate}] (0.3,0.15) -- (3.38,4.8);
\end{scope}
\begin{scope}[very thick,nearly transparent,blue,decoration={
    markings,
    mark=at position 0.73 with {\arrow{{Stealth}{Stealth}}}}
    ] 
    \draw[postaction={decorate}] (1.7,0.85) -- (3.38,3.37);
\end{scope}
\begin{scope}[very thick,blue,decoration={
    markings,
    mark=at position 1 with {\arrow{{Stealth}{Stealth}}}}
    ] 
    \draw[postaction={decorate}] (1,1) -- (0.25,5/8);
    \draw[postaction={decorate}] (3.3,3.3) -- (2.55,2.925);
    \draw[postaction={decorate}] (5.6,5.6) -- (4.85,5.225);
\end{scope}
\end{tikzpicture}
\end{center}
\caption{Slow (blue) and superslow (purple) dynamics on the critical manifolds $\mathcal{M}^{(i)}$, \mbox{$i=1, 2, 3$} \eqref{eq:critmanifolds} $(v,s,w)$-space for $q=0$.}
\label{fig:critman}
\end{figure}
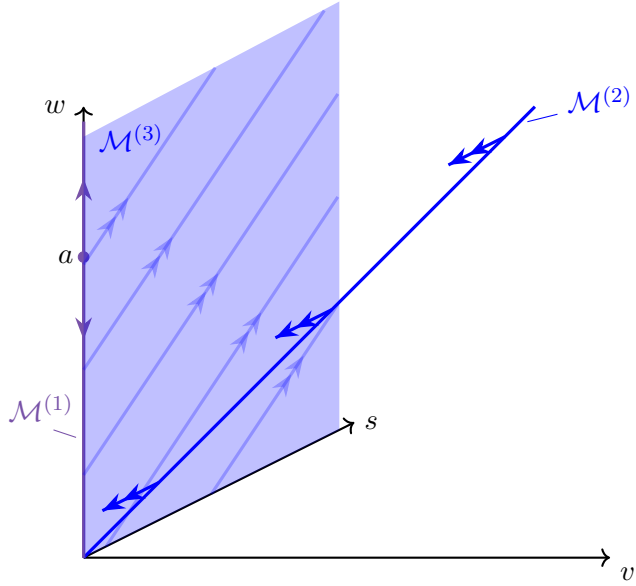

\subsection{Layer problem}
We recall that our sought-after travelling pulses correspond to solutions homoclinic to the equilibrium $\left( w, v, q, s \right) = \left( a, 0, 0, 0 \right)$. From the analysis carried out in Section \ref{sec:redprob}, we see that the equilibrium $\left( a, 0, 0, 0 \right)$ lies on the invariant manifold $\mathcal{M}^{(3)}$; however, the dynamics on $\mathcal{M}^{(3)}$ are linear \eqref{eq:slowflow_M3}. Therefore, $\mathcal{M}^{(3)}$ cannot contain a homoclinic orbit, and we conclude that any orbit homoclinic to $\left( a, 0, 0, 0 \right)$ must necessarily contain fast segments, indicating an excursion away from $\mathcal{M}^{(3)}$.\\
In the singular limit, the flow on $\mathcal{M}^{(3)}$ \eqref{eq:slowflow_M3} is decomposed into a first superslow segment on $\mathcal{M}^{(1)}$ from $p_1(a)=(a,0,0,0)$ to a point $(w^\ast, 0, 0, 0)$, followed by a slow transition on $\mathcal{M}^{(3),w^\ast}$ from $(w^\ast, 0, 0, 0)$ to $p_3(a,s^\ast)=(a,0,0,s^\ast)$ for particular choices of $w^\ast$ and $s^\ast$, see Figure \ref{fig:slow_super_slow}. The fast transition closing the loop would be a heteroclinic orbit in the layer problem \eqref{eq:laypb} from the equilibrium $p_3(a,s^\ast)$ to $p_1(a)$.\\

Our goal therefore is to determine the value of the constants $w^\ast$ and $s^\ast$ by investigating the layer problem, in particular by ensuring that there exists a unique fast connection from $p_3(a,s^\ast)$ to $p_1(a)$. A priori, obtaining this heteroclinic orbit in the three-dimensional layer problem \eqref{eq:laypb} would be prohibitive; however, we will exploit the structure of the layer problem to show that the target heteroclinic lies close to the union of two heteroclinic connections to and from $p_2(a)$, see Figure \ref{fig:fast_fig}. The existence of the pair of heteroclinic connections to and from $p_2(a)$ is determined in Lemma \ref{lemma:fc2} and Lemma \ref{lemma:fc1}.


\begin{figure}[b!]
\begin{center}
\begin{tikzpicture}
\draw[thick,->] (0,0) -- (7,0);
\node [below right] at (7,0) {$v$};
\draw[thick,->] (0,0) -- (3.6,1.8);
\node [right] at (3.6,1.8) {$s$};
\draw[thick,->] (0,0) -- (0,6);
\node [left] at (-0.1,6) {$w$};

\draw[thick,dotted,gray] (0,4) -- (3.6,5.8);
\begin{scope}[very thick,royalpurple,decoration={
    markings,
    mark=at position 0.5 with {\arrow{{Stealth}}}}
    ] 
    \draw[postaction={decorate}] (0,4) -- (0,2);
\end{scope}
\draw[fill,royalpurple] (0,2) circle [radius=2pt];
\node[left] at (0,2) {$w^\ast$};
\draw[fill,royalpurple] (0,4) circle [radius=2pt];
\node[left] at (0,4) {$a$};
\node[right] at (0,3.9) {\textcolor{royalpurple}{$p_1(a)$}};
\node[below] at (2,1) {$s^\ast$};
\draw[fill,blue] (2,5) circle [radius=2pt];
\node[right] at (2,4.9) {\textcolor{blue}{$p_3(a,s^\ast)$}};
\draw[thick,dotted,gray] (2,1) -- (2,5);
\begin{scope}[very thick,blue,decoration={
    markings,
    mark=at position 0.5 with {\arrow{{Stealth}{Stealth}}}}
    ] 
    \draw[postaction={decorate}] (0,2) -- (2,5);
\end{scope}
\end{tikzpicture}
\end{center}
\caption{Schematic representation of the slow (blue) and superslow (purple) dynamics of the travelling pulse. The dynamics take place on $\mathcal{M}^{(3)}$ \eqref{eq:slowflow_M3}: the superslow segment is contained in $\mathcal{M}^{(1)}$ \eqref{eq:superslowM1}, and the slow segment is contained in $\mathcal{M}^{(3),w_0}$ \eqref{eq:M3_w0}. See also Figure \ref{fig:critman}.}
\label{fig:slow_super_slow}
\end{figure}
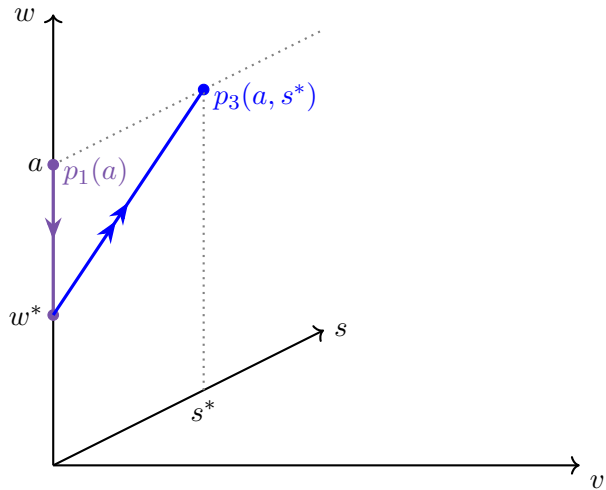



\begin{figure}[h!]
\begin{center}
\begin{tikzpicture}
\draw[thick,->] (0,0) -- (7,-2);
\node [below right] at (7,-2) {$v$};
\draw[thick,->] (0,0) -- (3.6,1.8);
\node [right] at (3.6,1.8) {$s$};
\draw[thick,->] (0,-3) -- (0,4);
\node [left] at (-0.1,4) {$q$};

\fill[nearly transparent,pastelgreen] (0,-1) -- (5,-2.42857) -- (5,-0.42857) -- (0,1);
\fill[nearly transparent,pastelgreen] (4.15,1) -- (0.55,1.5) -- (3.8,-1.3) -- (7.4,-1.8);

\begin{scope}[very thick,blue,decoration={
    markings,
    mark=at position 0.7 with {\arrow{{Stealth}{Stealth}}}}
    ] 
    \draw[postaction={decorate}] (0,0) -- (2.5,1.25);
\end{scope}
\draw[fill,royalpurple] (0,0) circle [radius=2pt];
\node[below left,royalpurple] at (0,0.1) {$p_1(a)$};
\draw[fill,blue] (2.5,1.25) circle [radius=2pt];
\node[left,blue] at (3,1.6) {$p_3(a,s^\ast)$};
\draw[fill,blue] (5,-1.42857) circle [radius=2pt];
\node[above right,darkpastelgreen] at (4.9,-2.22857) {$p_2(a)$};


\begin{scope}[very thick,darkpastelgreen,decoration={
    markings,
    mark=at position 0.6 with {\arrow{{Stealth}{Stealth}{Stealth}}}}
    ] 
    \draw[postaction={decorate}] 
    (5,-1.42857) .. controls (3,-1.5) and (1,-1.5) .. (0,0);
    \draw[postaction={decorate}] (2.5,1.25) .. controls  (2.3,0.8) and (2.5,0.2) .. (5,-1.42857);
\end{scope}
\end{tikzpicture}
\end{center}
\caption{Schematic representation of the fast dynamics associated to a travelling pulse solution to \eqref{eq:postscalingw}, in $(v,q,s)$-space for fixed $w=a$. The existence of the heteroclinic connection from $p_3(a,s^\ast)$ to $p_2(a)$ follows from Lemma \ref{lemma:fc1}; the existence of the heteroclinic connection from $p_2(a)$ to $p_1(a)$ follows from Lemma \ref{lemma:fc2}. The invariant planes $\Pi$ \eqref{eq:plane_Pi} and $\left\{\,s=0\,\right\}$ are indicated in green.}
\label{fig:fast_fig}
\end{figure}
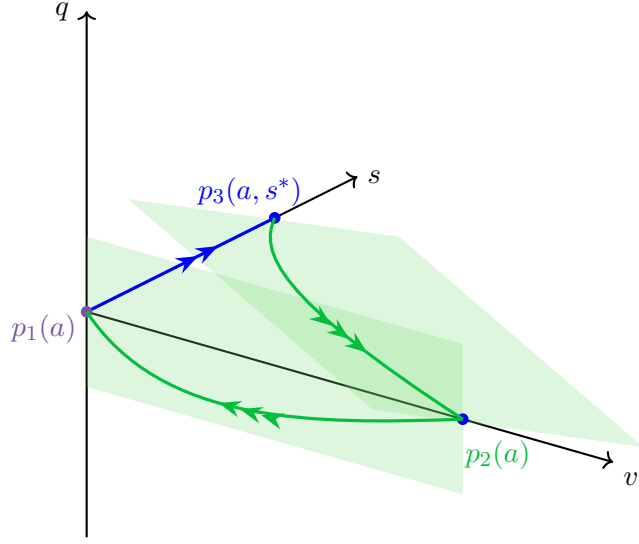

\begin{lemma} \label{lemma:fc2}
There exists a unique $c=c^\ast>0$ for which system \eqref{eq:laypb} admits a heteroclinic orbit from $p_2(a) =(a,a,0,0)$ to $p_1(a)=(a,0,0,0)$. This heteroclinic orbit lies in the hyperplane $\left\{ (w,v,q,s) \, : \, w=a, \ s=0 \right\}$. In particular,
\begin{equation}\label{eq:c_ast}
c^\ast=(a \, \theta_0)^{2/3}
\end{equation}
with $\theta_0 \approx 0.8615$.
\end{lemma}

\begin{proof}
The hyperplane $\left\{\,s=0\right\}$ is invariant under the flow of system \eqref{eq:laypb}. On this hyperplane, the layer problem reduces to
\begin{equation} \label{eq:laypbs0}
\begin{aligned}
\dot{v} &= c^3 q, \\
\dot{q} &= - \left( w - q - v \right) v^2 - c^3 q, 
\end{aligned}
\end{equation}
which coincides with the layer problem studied in \cite[Proposition 2.2]{Carter2018}. The existence of a unique value of $c=c^\ast=(a \, \theta_0)^{2/3}$ with $\theta_0 \approx 0.8615$ such that there is a heteroclinic orbit connecting the two equilibria $p_1(a)$ and $p_2(a)$
hence directly follows.
\end{proof}

\begin{lemma} \label{lemma:fc1}
There exists a unique $s= s^\ast = \frac{a}{\mathcal{D}}$ for which system \eqref{eq:laypb} admits a heteroclinic orbit from $p_3(a,s^\ast) = (a,0,0,s^\ast)$ to $p_2(a) = (a,a,0,0)$. This heteroclinic orbit lies in the hyperplane $\left\{ (w,v,q,s) \, : \, w=a,\,w-v-q-\mathcal{D} s = 0\, \right\}$.
\end{lemma}

\begin{proof}
We fix $w=a$, and observe that the line $\bigcup_{s} p_3(a,s) = \mathcal{M}^{(3)} \cap \left\{ w = a \right\}$ is invariant under the flow of the layer problem \eqref{eq:laypb}. For fixed $s=\bar{s}$, $p_3(a,\bar{s})$ is an equilibrium of the layer problem \eqref{eq:laypb} with eigenvalues
\begin{equation}\label{eq:eigenvalues_p3}
\lambda^c_3(a,\bar{s}) = 0\quad\text{and}\quad \lambda^{s,u}_3(a,\bar{s}) = -\frac{1}{2}c^{3/2}\left(-c^{3/2} \pm \sqrt{c^3 + 4 \mathcal{H} \bar{s}}\right).    
\end{equation}
The corresponding centre, stable, and unstable eigenvectors are given by
\begin{equation} \label{eq:evprop}
   \bm{\eta}^c = (0,0,1), \qquad \bm{\eta}^{s,u} = \left(-\lambda_{s,u},-\mathcal{H} \bar{s} + \lambda_{s,u},\frac{\mathcal{H} \bar{s}}{\mathcal{D}}\right).
\end{equation}
Performing a coordinate change $(v,q,s) = (0,0,\bar{s}) + w_c \bm{\eta}^c + w_s \bm{\eta}^s + w_u \bm{\eta}^u$, we obtain for the centre dynamics
\begin{equation}  \label{eq:laypbwc}
\dot{w}_c = \frac{(\mathcal{D} (\bar{s}+w_c)-a) (\lambda^s_3(a,\bar{s}) w_s+\lambda^u_3(a,\bar{s}) w_u)^2}{\mathcal{D}}, 
\end{equation}
from which we see that the plane $\left\{w_c = 0\right\}$ is globally invariant if and only if \mbox{$\bar{s} = \frac{a}{\mathcal{D}} =: s^\ast$}. Reverting to the original coordinates, this plane transforms to 
\begin{equation}\label{eq:plane_Pi}
\Pi = \left\{a - v- q - \mathcal{D}s = 0\right\};    
\end{equation}
note that this is $\mathcal{O}(\delta)$ close to the hyperplane $\left\{u=0, w=a\right\}$, cf.~\eqref{eq:definition_w_variable}.\\
We observe that both $p_2(a)$, $p_3(a,s^\ast) \in \Pi$. 
Since the eigenvalues $\lambda^{(i)}_2(a)$ of the equilibrium $p_2(a)$ are given by 
\begin{equation}\label{eq:eigenvalues_p2}
\lambda^{(1)}_2(a) = -c^3,\quad \lambda^{(2)}_2(a)= -\frac{\mathcal{H} a}{\mathcal{D}},\quad \lambda^{(3)}_2(a) = a^2,
\end{equation}
the stable manifold of $p_2(a)$ is two-dimensional. The stable eigenspace of $p_2(a)$ is spanned by $\left(-1,1,0\right)$ and $\left(-c^3 \mathcal{D}^2, a \mathcal{D} \mathcal{H}, c^3 \mathcal{D} - a \mathcal{H}\right)$. Inspection of these eigenvectors reveals that the stable eigenspace of $p_2(a)$ is contained in the plane $\Pi$; hence, the two planes coincide. The stable manifold $\mathcal{W}^s\left(p_2(a)\right)$ is therefore tangent to $\Pi$ at $p_2(a)$. However, since both $\mathcal{W}^s\left(p_2(a)\right)$ and $\Pi$ are two-dimensional manifolds that are invariant under the flow of \eqref{eq:laypb} and $p_2(a)$ is hyperbolic, it follows that $\mathcal{W}^s\left(p_2(a)\right)$ and $\Pi$ must coincide, as the stable manifold is (locally) unique.\\
The dynamics of \eqref{eq:laypb} on $\Pi = \mathcal{W}^s\left(p_2(a)\right)$ are given by
\begin{equation} \label{eq:dynamics_P}
\begin{aligned}
\dot{v} &= c^3 q, \\
\dot{q} &= \frac{\mathcal{H}}{\mathcal{D}}\left( a - q - v \right) v - c^3 q. 
\end{aligned}
\end{equation}
The equilibrium $(v,q) = (0,0)$ of \eqref{eq:dynamics_P}, representing $p_3(a,s^\ast)$, is of saddle type. Its one-dimensional unstable manifold lies per construction inside the stable manifold of $p_2(a)$, represented by the equilibrium $(a,0)$ of \eqref{eq:dynamics_P}, which is a sink. It follows that a heteroclinic orbit from $(0,0)$ to $(a,0)$ in \eqref{eq:dynamics_P} is guaranteed to exist.
\end{proof}




\subsection{Singular orbit} \label{sec:singorb}

We can combine the results obtained for the layer problem \eqref{eq:laypb} from Lemma \ref{lemma:fc2} and Lemma \ref{lemma:fc1} with the analysis of the reduced problems presented in Section \ref{sec:redprob} to obtain a singular structure for the homoclinic orbit to the desert state $(a,0,0,0)$. 
In the following, for any $\underline{w}$, $\overline{w}$ such that $\underline{w}<\overline{w}$ we refer to subsets of $\mathcal{M}^{(1)}$ and $\mathcal{M}^{(3),w_0}$ using the notation
\begin{equation}
\begin{aligned}
 \mathcal{M}^{(i)} (\underline{w},\overline{w}) &= \left\{ p_i(w) \, : \, w \in [\underline{w},\overline{w}] \right\}, \qquad i=1,2, \\
 \mathcal{M}^{(3),w_0} (\underline{w},\overline{w}) &= \left\{ p_3(w,s(w)) \, : \, w \in [\underline{w},\overline{w}] \right\}.
\label{eq:M1M3w}
\end{aligned}
\end{equation}

\begin{proposition}\label{thm:singsol}
 For any fixed value of $a$, $\mathcal{B}$, $\mathcal{D}$, $\mathcal{H}>0$ and given $c=c^\ast=(a \, \theta_0)^{2/3}$, there exists a unique singular skeleton orbit $\Phi_0$ homoclinic to $p_1(a)=(a,0,0,0)$ corresponding to a singular solution to system \eqref{eq:postscalingw}, composed of segments
of orbits of the layer problem \eqref{eq:laypb} and the reduced problems \eqref{eq:slowflow_M3} and \eqref{eq:superslowM1}. In particular, the singular solution $\Phi_0$ is given by the concatenation
\begin{equation} \label{eq:singhom}
  \Phi_0=\mathcal{M}^{(1)}(0, a) \cup \mathcal{M}^{(3),0}(0, a) \cup \phi_{f,1} \cup \phi_{f,2},
\end{equation}
where $\phi_{f,1}$ and $\phi_{f,2}$ are the heteroclinic orbits determined in Lemma \ref{lemma:fc2} and Lemma \ref{lemma:fc1}, respectively.
\end{proposition}


\begin{proof}
The concatenation of the two fast heteroclinic orbits $\phi_{f,1}$ and $\phi_{f,2}$ determines a fast heteroclinic connection from $p_3(a,s^\ast)$ to $p_1(a)$, with $s^\ast = \frac{a}{\mathcal{D}}$ as given in Lemma \ref{lemma:fc1}; Lemma \ref{lemma:fc2} determines the value of $c^\ast$ \eqref{eq:c_ast}.\\
From the definition of $\mathcal{M}^{(3),w_0}$ \eqref{eq:M3_w0}, it follows that $p_3(a,s^\ast) \in \mathcal{M}^{(3),0}$. From the fibration of $\mathcal{M}^{(3)}$ along $\mathcal{M}^{(1)}$, combined with the slow-fast structure of the flow \eqref{eq:slowflow_M3} on $\mathcal{M}^{(3)}$, it follows that the slow singular segment is composed of a superslow segment on $\mathcal{M}^{(1)}$ from $p_1(a,0,0,0)$ to the origin, followed by a slow segment on $\mathcal{M}^{(3),0}$ from the origin to $p_3(a,s^\ast)$.
\end{proof}
For a sketch of the singular skeleton $\Phi_0$, see Figure \ref{fig:full}.

\begin{figure}[H]
\begin{center}
\begin{tikzpicture}
\draw[thick,->] (0,0) -- (7,0);
\node [below right] at (7,0) {$v$};
\draw[thick,->] (0,0) -- (3.6,1.8);
\node [right] at (3.6,1.8) {$s$};
\draw[thick,->] (0,0) -- (0,6);
\node [left] at (-0.1,6) {$w$};

\draw[thick,dotted,gray] (0,4) -- (3.6,5.8);
\begin{scope}[very thick,royalpurple,decoration={
    markings,
    mark=at position 0.5 with {\arrow{{Stealth}}}}
    ] 
    \draw[postaction={decorate}] (0,4) -- (0,0);
\end{scope}
\draw[fill,blue] (0,0) circle [radius=2pt];
\node[left] at (0,0) {$0$};
\draw[fill,royalpurple] (0,4) circle [radius=2pt];
\node[left] at (0,4) {$a$};
\node[below right] at (0,4) {\textcolor{royalpurple}{$p_1(a)$}};
\node[below] at (2,1) {$s^\ast$};
\draw[fill,blue] (2,5) circle [radius=2pt];
\node[above left] at (2,4.8) {\textcolor{blue}{$p_3(a,s^\ast)$}};
\draw[thick,dotted,gray] (2,1) -- (2,5);
\begin{scope}[very thick,blue,decoration={
    markings,
    mark=at position 0.5 with {\arrow{{Stealth}{Stealth}}}}
    ] 
    \draw[postaction={decorate}] (0,0) -- (2,5);
\end{scope}
\draw[thick,dotted,gray] (0,4) -- (6,4); 
\draw[thick,dotted,gray] (4,0) -- (4,4); 
\node[below] at (4,0) {$a$}; 
\begin{scope}[very thick,darkpastelgreen,decoration={
    markings,
    mark=at position 0.7 with {\arrow{{Stealth}{Stealth}{Stealth}}}}
    ] 
    \draw[postaction={decorate}] (2,5) .. controls (2,4.7) and (2.5,4.5) .. (4,4);
\end{scope}
\begin{scope}[very thick,darkpastelgreen,decoration={
    markings,
    mark=at position 0.7 with {\arrow{{Stealth}{Stealth}{Stealth}}}}
    ] 
    \draw[postaction={decorate}] (4,4) .. controls (3,4) and (1,4) .. (0,4);
\end{scope}
\draw[fill,darkpastelgreen] (4,4) circle [radius=2pt];
\node[above right] at (4,4) {\textcolor{darkpastelgreen}{$p_2(a)$}};
\end{tikzpicture}
\end{center}
\caption{The singular skeleton orbit $\Phi_0$ \eqref{eq:singhom} given in Proposition \ref{thm:singsol}, obtained by matching superslow (purple), slow (blue) and fast (green) orbits of the reduced problems \eqref{eq:superslowM1}, \eqref{eq:slowflow_M3} and the layer problem \eqref{eq:laypb}.}
\label{fig:full}
\end{figure}
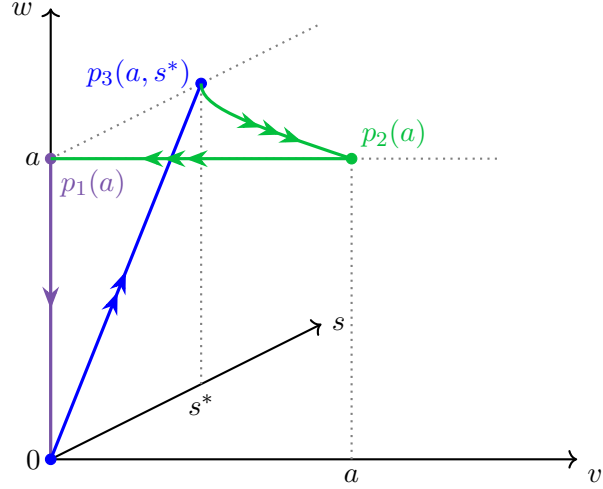

\subsection{Persistence} \label{sec:persist}

The singular skeleton orbit $\Phi_0$ \eqref{eq:singhom} provides the backbone for the main result of this paper: the existence of a homoclinic solution in system \eqref{eq:postscalingw}. The persistence of $\Phi_0$ for $0<\delta \ll 1$ is given in the following Theorem.

\begin{theorem} \label{thm:pertsol}
Let $0<\delta \ll 1$ be sufficiently small, en let $a$, $\mathcal{B}$, $\mathcal{D}$, $\mathcal{H}$ be fixed $\mathcal{O}(1)$ and positive. There exists a unique value $c=c^\ast_\delta$ for which a unique orbit $\Phi_\delta$ to \eqref{eq:postscalingw} exists that is homoclinic to $p_1(a) = (a,0,0,0)$. Moreover, $\Phi_\delta$ is $\mathcal{O}(\delta)$ close to $\Phi_0$ \eqref{eq:singhom}, and $c^\ast_\delta$ is $\mathcal{O}(\delta)$ close to $c^\ast$ \eqref{eq:c_ast}.
\end{theorem}


\begin{proof}
Our goal is to show that the singular skeleton orbit $\Phi_0$ \eqref{eq:singhom} persists for $0<\delta \ll 1$ as a solution to \eqref{eq:postscalingw} that is $\mathcal{O}(\delta)$ close to $\Phi_0$.\\
First, we focus on the dynamics on the manifold $\mathcal{M}^{(3)}$, which is invariant under the flow of \eqref{eq:postscalingw}. The dynamics on $\mathcal{M}^{(3)}$ are given by \eqref{eq:slowflow_M3}. 
System \eqref{eq:slowflow_M3} is linear, and can thus be solved explicitly, yielding
\begin{equation}\label{eq:w_s_explicit_sols}
    w(\sigma)=a +c_1 \mathcal{D} e^{\frac{\sigma}{\mathcal{D}}} +(c_2-c_1 \mathcal{D}) e^{\frac{\delta  \sigma }{1+\delta}}, \qquad s(\sigma) = c_1 e^{\frac{\sigma}{\mathcal{D}}}.
\end{equation}
Equivalently, solving
\[
\frac{dw}{ds} = 1+ \frac{\delta}{1+\delta} \frac{\left( w - \mathcal{D} s - a \right)}{s},
\]
we obtain
\begin{equation}\label{eq:w(s)_explicit_sol}
    w(s)=a+\mathcal{D} s+k_1 (s(1+\delta))^{\frac{ \delta\mathcal{D} }{1+\delta}}=:w_3(s;k_1).
\end{equation}
Therefore, taking a neighbourhood of the point $p_3(a,s^\ast)$ on $\mathcal{M}^{(3)}$ as follows
\begin{equation} \label{eq:sigmasec}
    \Sigma := \left\{ (w,v,q,s)  \, : \, v=q=0, \, w=a, \, |s-s^\ast|< \mu \right\}
\end{equation}
for $\mu$ sufficiently small, we have that the $\alpha$-limit set of $\Sigma$ solely consists of the point $p_1(a)$. The section $\Sigma$ can thus be interpreted as the exit section of the intersection of the unstable manifold of $p_1(a)$ with $\mathcal{M}^{(3)}$ given by
\begin{equation} \label{eq:Wumu}
    \mathcal{W}^u_{\delta,\mu}(p_1(a)) := \left\{ (w,v,q,s) \, : \, v=q=0, \, w_3(s,k_1^+(\mu)) < w < \, w_3(s,k_1^-(\mu)), \, s > 0 \right\},
\end{equation}
where $w_3(s,k_1^\pm(\mu)) = a+ \mathcal{D} s \left(1-\left(\frac{a \pm \mu \mathcal{D}}{\mathcal{D} s}\right)^{1-\frac{\delta \mathcal{D}}{1+\delta}}\right)$, see also Figure \ref{fig:proof2D}. We remark that by definition $\mathcal{W}^u_{\delta,\mu}(p_1(a)) \cap \left\{ w=a \right\} = \Sigma$.

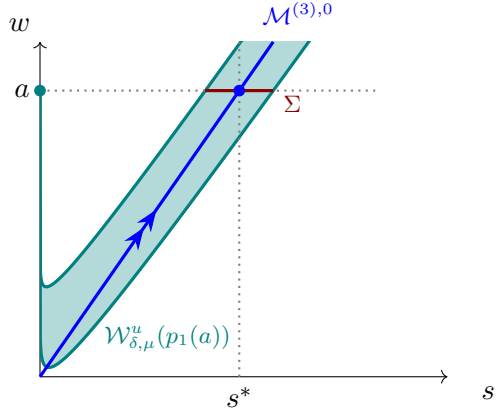
\begin{figure}[H]
\begin{center}
\begin{tikzpicture}
\begin{axis}[
    clip mode=individual,
    axis lines = middle,
    axis line style={->},
    xlabel style={at={(axis description cs:1.1,0)},anchor=north},
    ylabel style={at={(axis description cs:-0.05,1.1)},anchor=north},
    xlabel = {$s$},
    ylabel = {$w$},
    xticklabel=\empty,
    yticklabel=\empty,
    xmin=0, xmax=0.6,
    ymin=0, ymax=1.545]
 
\addplot [name path = A,
    domain = 0:0.5,
    very thick,
    teal,
    samples = 1000] {1.32 - 1.62037*x^(0.0445545) + 4.5*x}; 
 
\addplot [name path = B,
    domain = 0:0.5,
    very thick,
    teal,
    samples = 1000] {1.32 - 1.16617*x^(0.0445545) + 4.5*x}; 

\addplot [teal!30] fill between [of = A and B, soft clip={domain=0:0.5}];
\node[right,teal] at (0.08,0.18) {\footnotesize{$\mathcal{W}^u_{\delta,\mu}(p_1(a))$}};

\draw [thick,dotted,gray] (axis cs:{0,1.32}) -- (axis cs:{0.5,1.32});
\node[left] at (0,1.32) {$a$}; 
\draw [thick,dotted,gray] (axis cs:{0.293333,0}) -- (axis cs:{0.293333,1.545});
\node[below] at (0.293333,0) {$s^\ast$};
\draw [very thick, darkred] (axis cs:{0.243333,1.32}) -- (axis cs:{0.343333,1.32});
\node[right,darkred] at (0.343333,1.26) {\footnotesize{$\Sigma$}};

\begin{scope}[very thick,blue,decoration={
    markings,
    mark=at position 0.5 with {\arrow{{Stealth}{Stealth}}}}
    ] 
    \draw[postaction={decorate}] (0,0) -- (0.343333,1.545);
\end{scope}
\node[above,blue] at (0.38,1.58) {\footnotesize{$\mathcal{M}^{(3),0}$}};
\draw[fill,teal] (0,1.32) circle [radius=2pt];
\draw[fill,blue] (0.293333,1.32) circle [radius=2pt];
 
\end{axis}
\end{tikzpicture}
\end{center}
\caption{Sketch of the dynamics on $\mathcal{M}^{(3)}$ as described in the proof of Theorem \ref{thm:pertsol}. The teal set corresponds to $\mathcal{W}^u_{\delta,\mu}(p_1(a))$ as defined in \eqref{eq:Wumu}, the red section represents $\Sigma$ as in \eqref{eq:sigmasec}, and the blue curve shows the set $\Gamma$ defined in \eqref{eq:Gamma}. The teal and blue points correspond to $p_1(a)$ and $p_3(a,s^\ast)$, respectively.}
\label{fig:proof2D}
\end{figure}

The concatenation of the pair of heteroclinics determined in Lemma \ref{lemma:fc2} and Lemma \ref{lemma:fc1} provides a heteroclinic connection between $p_1(a)$ and $p_3(a,s^\ast)$ via $p_2(a)$. We first observe that the argument used in the proof of Lemma \ref{lemma:fc2} remains valid for other values of $w$, with associated speed value $c = c^\ast_w = (w \, \theta_0)^{2/3}$. Furthermore, the construction used in the proof of Lemma \ref{lemma:fc1} also carries over for other values of $w$, as an invariant hyperplane containing a heteroclinic from $p_3(w,s_w^\ast)$ to $p_2(w)$ can be found when $s = s_w^\ast = \frac{w}{\mathcal{D}}$.\\
Consequently, let us consider the neighbourhood of $p_1(a)$ given by $\mathcal{M}^{(1)} (a-\mu \mathcal{D},a+\mu \mathcal{D})$ and track its singular fast flow backwards until it reaches again $\mathcal{M}^{(3)}$. Our aim is to show that on this hyperplane $\mathcal{W}^s(\mathcal{M}^{(1)} (a-\mu \mathcal{D},a+\mu \mathcal{D}))$ intersects $\Sigma$ transversely in the point $p_3(a,s^\ast)$, which in turn will guarantee the persistence of a perturbed solution $\mathcal{O}(\delta)$-close to the singular one.\\
We start by observing that $\mathcal{M}^{(1)} (a-\mu \mathcal{D},a+\mu \mathcal{D})$ admits a strongly stable foliation, which by construction we can follow backwards until $\mathcal{M}^{(2)}$ - a manifold of normally hyperbolic equilibria of the layer problem \eqref{eq:laypb}. This set thus corresponds to the 2-dimensional manifold $\mathcal{W}^u(\mathcal{M}^{(2)} (a-\mu \mathcal{D},a+\mu \mathcal{D}))$. This manifold transversely intersects the 2-dimensional manifold $\mathcal{W}^s(\mathcal{M}^{(2)} (a-\mu \mathcal{D},a+\mu \mathcal{D}))$, which by construction (we remark here that $w$ acts as a parameter for the layer problem) intersects $\mathcal{M}^{(3)}$ in the curve given by
\begin{equation} \label{eq:Gamma}
    \Gamma = \left\{ (w,v,q,s) \, : \, v=q=0, \, w \in [a-\mu \mathcal{D},a+\mu \mathcal{D}], \, s \in [s^\ast - \mu, s^\ast + \mu] \right\}.
\end{equation}
The curve $\Gamma$ intersects $\Sigma$ in the unique point $p_3(a, s^\ast)$. Thanks to Fenichel theory \cite{Fe79} and the Exchange Lemma \cite{JKK96, Kuehn_2015}, analogously tracking $\mathcal{M}^{(1)} (a-\mu \mathcal{D},a+\mu \mathcal{D})$ for $0 < \delta \ll 1$ until $\mathcal{M}^{(3)}$, will lead to an intersection $\Sigma$ in a point $p_3^\ast(a,s^\ast_\delta)$ which is $\mathcal{O}(\delta)$-close to $p_3(a,s^\ast)$; the corresponding final piece of the orbit ``returning'' to $w=a$ on $\mathcal{M}^{(3)}$ is then given by
\begin{equation} \label{eq:w3ast}
 w_3^\ast(s)=a+\mathcal{D} s \left(1-\left(\frac{s^\ast_\delta}{s}\right)^{1-\frac{ \delta \mathcal{D}}{1+\delta}}\right).
\end{equation}

\noindent
We note that for every fixed $0 < \delta \ll 1$ we can choose $\mu$ such that  $|s_\delta^\ast-s^\ast| \leq \mu$ and $w$ is nonnegative in  $\mathcal{W}^u_{\delta,\mu}(p_1(a))$ as long as $\mu$ is $\mathcal{O}(\delta \, \mathrm{log}(\delta))$ and satisfies 
 $\mu \leq \frac{1}{\mathcal{D}}\left(\left(\frac{1+\delta}{\delta  \mathcal{D}^{\frac{\delta \mathcal{D}}{1+\delta }}}\right)^{\frac{1+\delta}{1+\delta(1-\mathcal{D})}}-a\right)$.

%
%
%
%

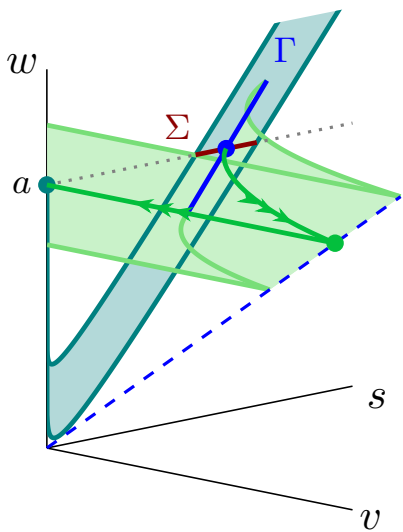
\begin{figure}
\begin{center}
\scalebox{1.5}{
\begin{tikzpicture}
\begin{axis}[view={45}{13},
    clip mode=individual,
    axis lines*=center,
    xlabel style={at={(axis description cs:0.53,0.03)},anchor=north},
    ylabel style={at={(axis description cs:0.54,0.27)},anchor=north},
    zlabel style={at={(axis description cs:-0.04,0.84)},rotate=-90,anchor=south},
    xlabel=$v$,
    ylabel=$s$,
    zlabel=$w$,
    xticklabel=\empty,
    yticklabel=\empty,
    xmin=0, xmax=1.4,
    ymin=0, ymax=0.5,
    zmin=0, zmax=1.9] 

\addplot3 [name path = A,
        teal,
        very thick,
        domain=0:0.5,
        domain y=0:0,
        samples=1000,
        samples y=1,
    ] (
{0},
{x},
{1.32 - 1.62037*x^(0.0445545) + 4.5*x}
);
 
\addplot3 [name path = B,
        teal,
        very thick,
        domain=0:0.5,
        domain y=0:0,
        samples=1000,
        samples y=1,
    ] (
{0},
{x},
{1.32 - 1.16617*x^(0.0445545) + 4.5*x}
);

\pgfmathdeclarefunction{myfunct}{1}{%
  \pgfmathparse{-0.2 + (0.9*(0.2 + 0.222222*#1))/(0.9 + x)}%
}

\addplot3 [name path = C,
       very thick,
       pastelgreen,
       domain=0.001:1.02,
       domain y=0:0,
       samples y=1,
       samples=50,
   ]  (x,{myfunct(1.02)},1.02);

\addplot3 [name path = D,
       very thick,
       pastelgreen,
       domain=0.001:1.62,
       domain y=0:0,
       samples y=1,
       samples=50,
   ]  (x,{myfunct(1.62)},1.62);

\addplot3 [name path = E,
        very thick,
        pastelgreen,
        domain=0.001:1.02,
        domain y=0:0,
        samples y=50,
        samples=50,
    ]  ({x},{0},{1.02});

\addplot3 [name path = F,
        very thick,
        pastelgreen,
        domain=0.001:1.62,
        domain y=0:0,
        samples y=50,
        samples=50,
    ]  ({x},{0},{1.62});

\addplot3 [teal!30] fill between [of = A and B];
\addplot3 [pastelgreen!40] fill between [of = C and D];
\addplot3 [pastelgreen!40] fill between [of = E and F];

\draw [thick,dotted,gray] (axis cs:{0,0,1.32}) -- (axis cs:{0,0.5,1.32});
\node[left] at (0,0,1.32) {\footnotesize{$a$}}; 
\draw [very thick, darkred] (axis cs:{0,0.243333,1.32}) -- (axis cs:{0,0.343333,1.32});
\node[left,darkred] at (0,0.26,1.45) {\footnotesize{$\Sigma$}};
\node[above,blue] at (0,0.39,1.62) {\footnotesize{$\Gamma$}};
\draw [thick,dashed,blue] (axis cs:{0,0,0}) -- (axis cs:{1.62,0,1.62});

\draw[fill,teal] (0,0,1.32) circle [radius=2pt];
\draw[fill,blue] (0,0.293333,1.32) circle [radius=2pt];
\draw[fill,darkpastelgreen] (1.32,0,1.32) circle [radius=2pt];
\draw[very thick,blue] (axis cs:{0,0.226667,1.02}) -- (axis cs:{0,0.36,1.62});

\begin{scope}[very thick,darkpastelgreen,decoration={
    markings,
    mark=at position 0.7 with {\arrow[line width=0.16mm]{{Stealth}{Stealth}{Stealth}}}}
    ] 
    \draw[postaction={decorate}] (1.32,0,1.32) -- (0,0,1.32);
\end{scope}
\begin{scope}[very thick,darkpastelgreen,decoration={
    markings,
    mark=at position 0.7 with {\arrow[line width=0.16mm]{{Stealth}{Stealth}{Stealth}}}}
    ] 
    \draw[postaction={decorate}] (0,0.293333,1.32) .. controls (0.2,0.203636,1.32) and (0.7,0.0775,1.32) .. (1.32,0,1.32);
\end{scope}

\end{axis}
\end{tikzpicture}    
}
\end{center}
\caption{Sketch of the dynamics described in the proof of Theorem \ref{thm:pertsol} in $(v,\,s,\,w)$-space, for fixed $q=0$. The teal set corresponds to $\mathcal{W}^u_{\delta,\mu}(p_1(a))$ as defined in \eqref{eq:Wumu}, the red section represents $\Sigma$ as in \eqref{eq:sigmasec}, and the blue solid (dashed) curve shows the set $\Gamma$ as in Eq.~\eqref{eq:Gamma} ($\mathcal{M}^{(2)}$ in Eq.~\eqref{eq:critmanifolds}). The teal, green, and blue points correspond to $p_1(a)$, $p_2(a)$, and $p_3(a,s^\ast)$, respectively, and the green section reproduces the unstable eigenspace of a neighbourhood of $\mathcal{M}^{(3),0}(a-\kappa,a+\kappa))$ as defined in \eqref{eq:M1M3w}.}
\label{fig:proof}
\end{figure}

\noindent
Therefore, the orbit homoclinic to $p_1(a)$ perturbs to a solution $\Phi_\delta$ for $0 < \delta \ll 1$ which is $\mathcal{O}(\delta)$ close to $\Phi_0$, as claimed.
\end{proof}

\begin{proof}[Proof of Theorem \ref{thm:main}]
The statement of Theorem \ref{thm:main} directly follows from Theorem \ref{thm:pertsol}, observing that the travelling pulse solution $\left( U, V, S \right) (x,t)$ to system \eqref{eq:ndimmod} corresponds to the homoclinic orbit $\Phi_\delta$ in system \eqref{eq:postscalingw}. Reverting the rescalings given in \eqref{eq:rescaling} with $c = c^\ast$ \eqref{eq:c_ast}, we obtain the leading order expression for the pulse wave speed
\[
 \mathcal{C}=\left( \frac{\mathcal{A}^2 \theta_0^2}{\varepsilon} \right)^{1/3} + \mathcal{O}(1).
\]
\end{proof}

\section{Numerical continuation} \label{sec:numcont}
In order to illustrate the validity of our analysis in Section \ref{sec:Existence}, culminating in Theorem \ref{thm:main}, and to provide insight into the role of the autotoxicity parameters $\mathcal{H}$ and $\mathcal{D}$ on the shape and dynamics of vegetation pulse patterns, we carry out a numerical analysis of system \eqref{eq:prescaling} by numerical continuation using the software AUTO \cite{Doedel1981}, to determine the bifurcation structure of pulse solutions to system \eqref{eq:prescaling}. The results of this numerical analysis are shown in Figures \ref{fig:bifdiag}, \ref{fig:profiles} and \ref{fig:orbits}. All solutions were numerically determined on a domain of size $L=1000$ with periodic boundary conditions.\\
The first row of Figure \ref{fig:bifdiag} shows the relation between the wave speed $\mathcal{C}$ and the parameters $\mathcal{A}$ (panel (a)), $\mathcal{H}$ (panel (b)), and $\mathcal{D}$ (panel (c)). In particular, panel (a) shows a saddle-node bifurcation of travelling pulses; an equivalent bifurcation occurs in the toxicity-free version of \eqref{eq:prescaling} \cite[Figure 8 (a)]{Carter2018}. For the upper branch, an increase of the precipitation rate $\mathcal{A}$ leads to an increase in the migration speed $\mathcal{C}$, which is in line with the leading order expression for $\mathcal{C}$ as a function of $\mathcal{A}$ as stated in Theorem \ref{thm:main}, indicated as a red dashed curve in Figure \ref{fig:bifdiag} (a). A solution on the upper branch for $\mathcal{A}=1.2$ and $\mathcal{C} = 5.8$, indicated by a black dot, is used as a starting point for the other continuations.\\
Panels (b) and (c) show that, for the chosen solution, the wave speed does not measurably depend on the value of $\mathcal{H}$ and $\mathcal{D}$; this is again in line with the statement of Theorem \ref{thm:main}. Note that, in all panels, $\varepsilon = 0.005$ is fixed. Apart from panel (a), this implies that $\delta = \varepsilon \mathcal{C} = 0.029$ is fixed as well. The strong agreement with the statement of Theorem \ref{thm:main} suggests that $\delta=0.029$ is indeed sufficiently small, except near the saddle-node bifurcation point $\mathcal{A} = 0.45$, $\mathcal{C} = 1.90$. The pulse solutions on the lower branch of panel (a) are characterised by a $U$-profile
asymptotically close to the equilibrium value $U_\ast=\mathcal{A}$. Since these profiles have little ecological relevance and numerically turn out to be unstable, we have therefore excluded these solutions from our analysis; see also \cite[Remark 1.3]{Carter2018}.\\
In the second and third row of Figure \ref{fig:bifdiag}, we show the influence on the toxicity parameters $\mathcal{H}$ and $\mathcal{D}$ on the maximum value of the field variables $U$, $V$ and $S$. An increase in $\mathcal{H}$ leads to a monotonic increase of the maximum value of the surface water density $U_\text{max}$ and of the maximum value of the toxicity $S_\text{max}$, whereas the maximum value of the biomass $V_\text{max}$ decreases monotonically. This behaviour is consistent with the fact that an increment of $\mathcal{H}$ corresponds to a higher plant sensibility to toxicity, leading to a lower biomass amplitude. The effect of increasing $\mathcal{D}$ is reciprocal to that of $\mathcal{H}$; this is consistent with the fact that an increase in $\mathcal{D}$ corresponds to a slower growth of the toxicity $S$ (cf.~System \eqref{eq:ndimmod}), reducing the stress of vegetation induced by the toxicity.\\

\begin{figure}[h!]
	\centering
	\includegraphics[width=1\textwidth]{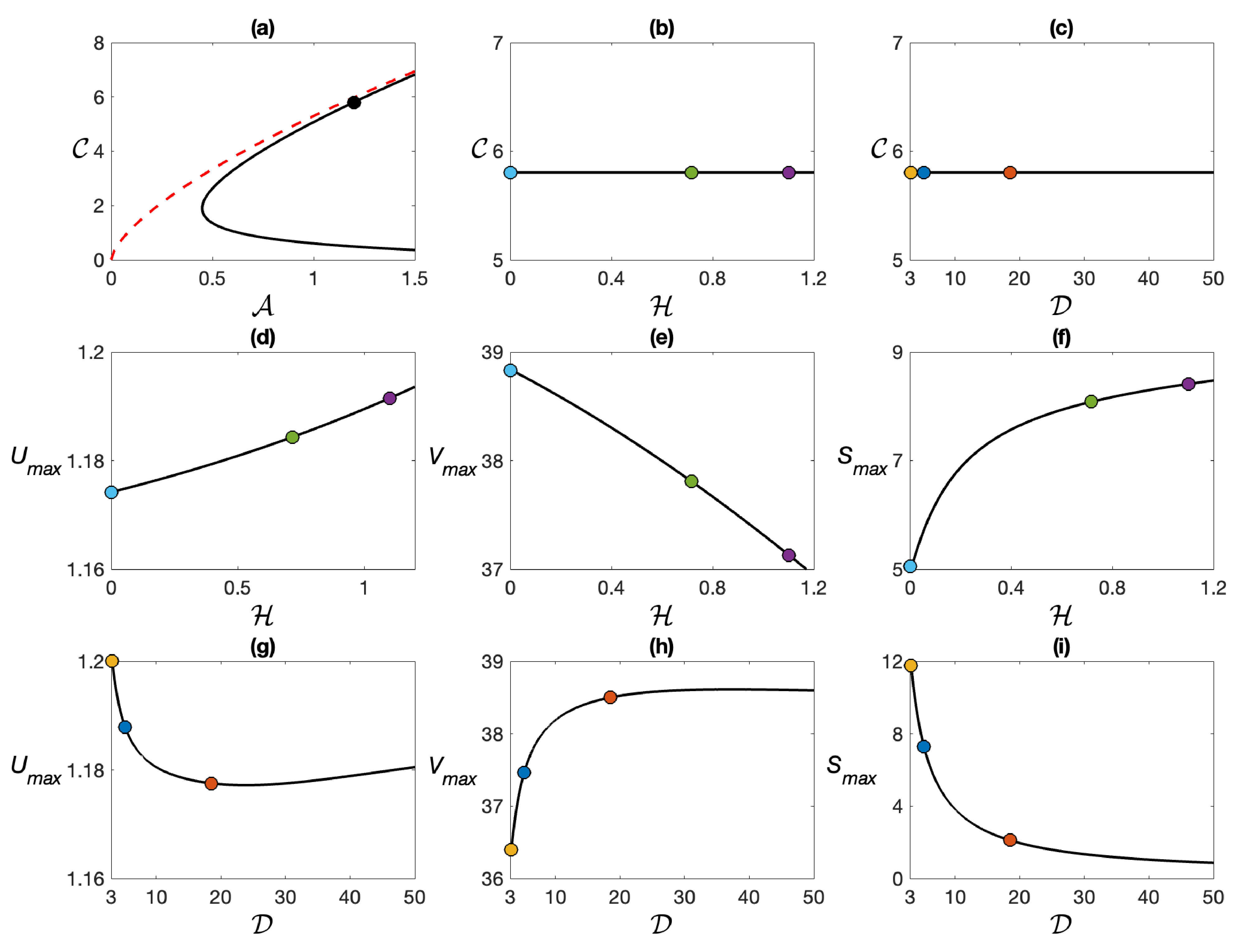}
        \caption{Bifurcation diagrams obtained by performing numerical continuation of system \eqref{eq:ndimmod} with AUTO varying (a) $\mathcal{A}$, (b, d-f) $\mathcal{H}$, and (c, g-i) $\mathcal{D}$. Panels (d-f, g-i) show the dependence of the field variables $(U,V,S)$ on the above parameters, whereas panels (a-c) depict how the migration speed $\mathcal{C}$ is affected by $\mathcal{A}$, $\mathcal{H}$, and $\mathcal{D}$, respectively. The red dashed line in panel (a) represents the migration speed derived analytically in Theorem \ref{thm:main}.\\
        The starting point for all the figures is the same as the one fixed in Figure \ref{fig:num} ($\mathcal{A}=1.2$, $\mathcal{B}=0.45$, $\varepsilon=0.005$, $\mathcal{D}=4.5$ and $\mathcal{H}=1$); subsequently, the corresponding control parameter is varied. All the results in panels (b),(c),(d)-(i) are related to the same point in panel (a) (top branch with $\mathcal{A}=1.2$ and $\mathcal{C}=5.8$), indicated by a black dot, since $\mathcal{A}$ is always fixed in all the other panels and $\mathcal{C}$ is not affected by $\mathcal{H}$ and $\mathcal{D}$.
        \\
        }
        \label{fig:bifdiag}
\end{figure}

In Figure \ref{fig:profiles}, we show the profiles of the individual state variables $U$, $V$ and $S$ for three fixed values of the parameters $\mathcal{H}$ and $\mathcal{D}$, corresponding to the coloured circles shown in Figure \ref{fig:bifdiag}. In panels (a-c), we observe that an increase of the plant sensitivity to autotoxicity, measured by $\mathcal{H}$, leads to a narrowing of the biomass pulse and a decrease of its amplitude, whereas the amplitude of the toxicity pulse increases. In contrast, slowing down the toxicity dynamics by increasing $\mathcal{D}$ (panels (d-f)) increases the with of the biomass pulse and lowers the amplitude of the toxicity component.\\


\begin{figure}[h!]
	\centering
	\includegraphics[width=1\textwidth]{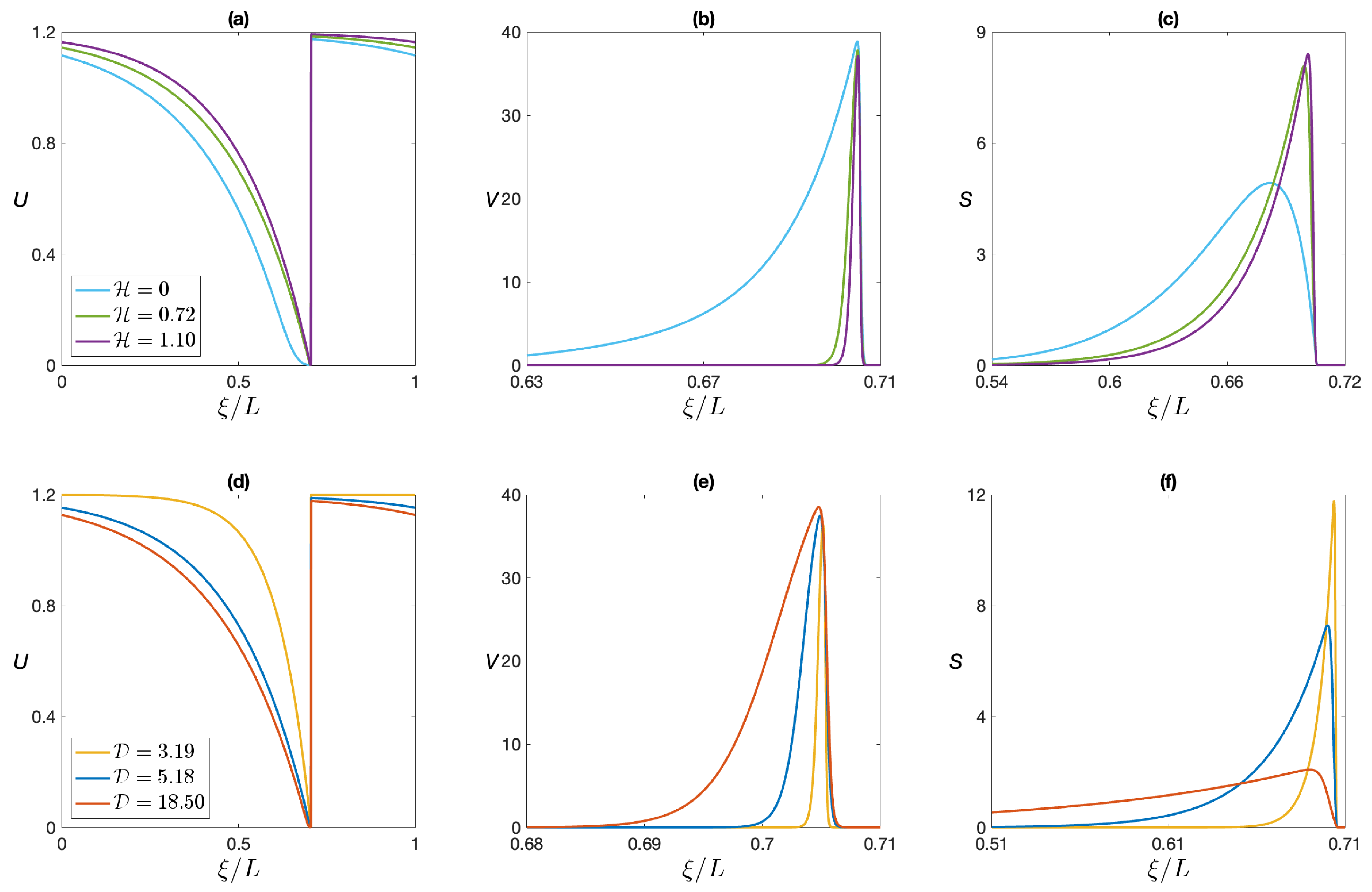}
        \caption{Profiles of the field variables (first column: $U$, second column: $V$, third column: $S$) in terms of the length-rescaled spatial variable $\xi/L$, for different values of $\mathcal{H}$ (panels (a-c)) and $\mathcal{D}$ (panels (d-f)) obtained by performing numerical continuation of system \eqref{eq:ndimmod} with AUTO. Different colours correspond to different parameter configurations as shown in Figure \ref{fig:bifdiag}.\\
        }
        \label{fig:profiles}
\end{figure}

In order to compare the numerical solutions depicted in Figure \ref{fig:profiles} with the analytical results obtained in the asymptotic case $0 < \delta \ll 1$ (see Figures \ref{fig:fast_fig}-\ref{fig:full}), we plot in Figure \ref{fig:orbits} the profiles as homoclinic orbits projected onto $(v,s,w)$-space (first column) and onto $(v,q,s)$-space (second column), with colours corresponding to the solutions shown in Figure \ref{fig:bifdiag} and Figure \ref{fig:profiles}. We observe that the value of $\mathcal{H}$ has little to no influence on the shape of the homoclinic, which corresponds to the observation that the singular skeleton described in Proposition \ref{thm:singsol} does not depend on $\mathcal{H}$. In fact, the value of $\mathcal{H}$ only occurs in the eigenvalues of $p_2$ \eqref{eq:eigenvalues_p2} and $p_3$ \eqref{eq:eigenvalues_p3}, determining the dynamics in the neighbourhood of the stable/unstable manifolds of the equilibria, which disappear in the singular limit $\delta \downarrow 0$. However, the case $\mathcal{H}=0$ is special, see the Discussion.\\
In contrast to the lack of influence of $\mathcal{H}$, the value of $\mathcal{D}$ does significantly change the shape of the homoclinic. Decreasing $\mathcal{D}$ (panels (c,d)) increases the value of $s^\ast = \frac{a}{\mathcal{D}}$ (cf.~Lemma \ref{lemma:fc1}), shifting the location of the equilibrium $p_3(a,s^\ast)$ further away from the $q$-axis. Note that the position of $p_2(a)$ \eqref{eq:equilibria} does not depend on $\mathcal{D}$.




    

\begin{figure}[h!]
	\centering
	\includegraphics[width=1\textwidth]{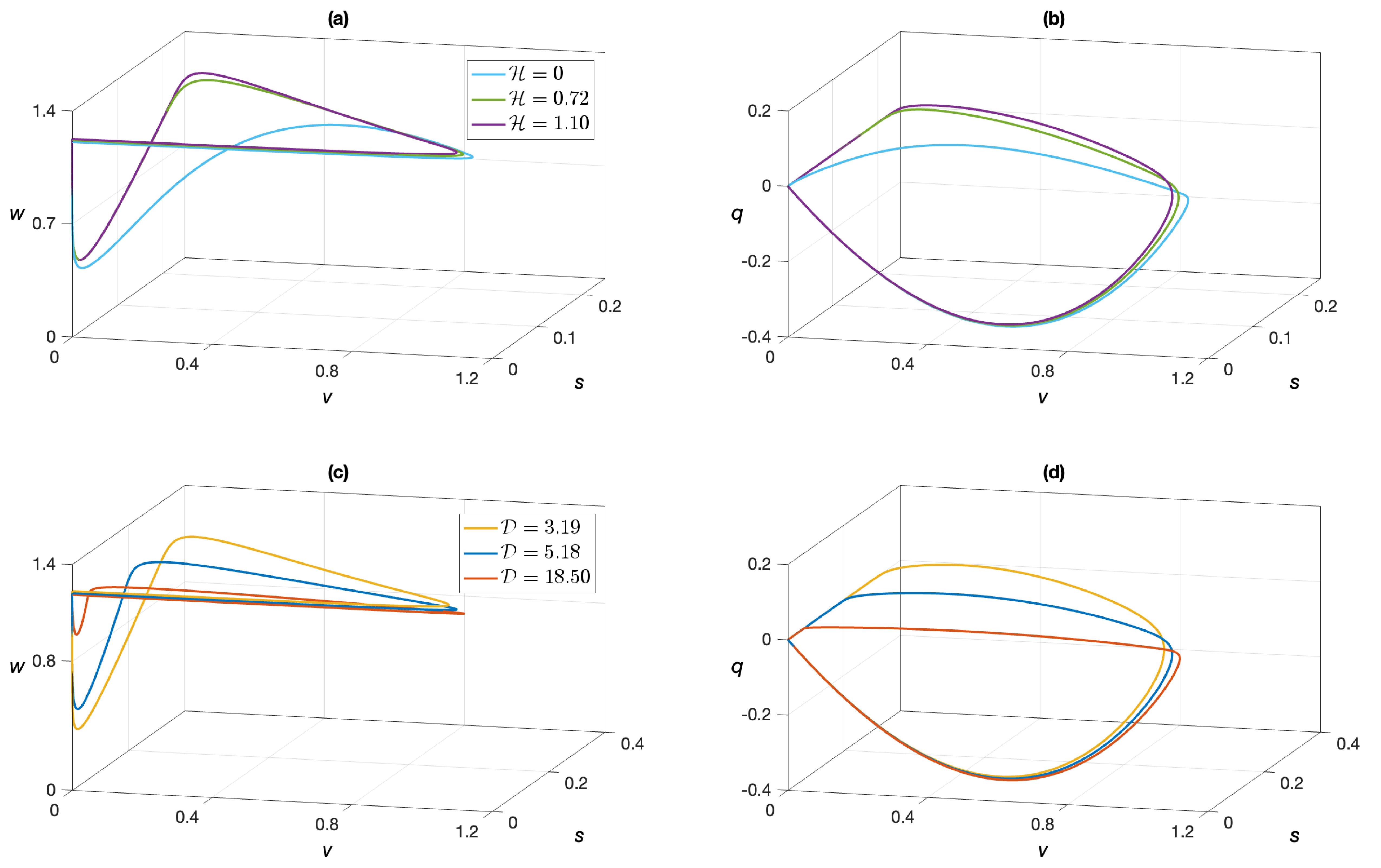}
        \caption{Homoclinic orbits obtained by performing numerical continuation of system \eqref{eq:ndimmod} with AUTO in ($v,s,w$)-space (panels (a,c)) and in ($v,s,q$)-space (panels (b,d)). Different colours correspond to different parameter configurations as shown in Figure \ref{fig:bifdiag}.}
        \label{fig:orbits}
\end{figure}

\section{Discussion}\label{sec:Conclusion}

In this manuscript, we have studied the influence of autotoxicity on far-from-threshold vegetation pulses in sloped semi-arid environments, extending the Klausmeier model with a third component that models the autotoxicity density. After employing a suitable rescaling, we have used geometric singular perturbation theory to establish the existence of a homoclinic orbit to the desert state, whose spatial profile corresponds to numerically observed travelling pulses.\\

The influence of autotoxicity is quantified by two parameters: $\mathcal{H}$, measuring biomass sensitivity to autotoxicity, and $\mathcal{D}$, measuring the autotoxicity decay rate. We observe that, although the shape of the pulse profile (strongly) depends on $\mathcal{H}$ and $\mathcal{D}$ (cf.~Figure \ref{fig:profiles}), the pulse propagation speed does not depend on $\mathcal{H}$ or $\mathcal{D}$, at least to leading order in $\delta$ (cf.~Theorem \ref{thm:main} and Figure \ref{fig:bifdiag} (b)-(c)). That is, in the model studied in this paper, the presence of autotoxicity does not influence the observed dynamics. This is in direct contrast with the results of \cite{Iuorio2021}, where equivalent analytical techniques were used in the context of a similar model to conclude that autotoxicity has a direct influence on far-from-threshold pulse dynamics. In the same context of \cite{Iuorio2021}, equivalent conclusions were drawn in \cite{Marasco2013,Marasco2014}; in addition, the influence of autotoxicity on the dynamics of vegetation patterns has also been observed in other contexts \cite{Iuorio2023pre,Iuorio2023preII}. A crucial difference between the latter models and the model studied in this paper, is that in the original Klausmeier model as studied e.g. in \cite{Carter2018} and in the current paper, water movement is driven by \emph{advection}, whereas in the models studied in \cite{Iuorio2021,Marasco2013,Marasco2014,Iuorio2023pre,Iuorio2023preII}, water movement is driven by \emph{diffusion}. Generalising this meta-observation, one could hypothesise that autotoxicity does influence vegetation dynamics in environments where water movement is diffusion driven, but does not influence vegetation dynamics in environments where water movement is advection driven (such as on sloped topographies). From a more general mathematical perspective, the effect of coupling an ODE to an existing reaction-advection-diffusion system on the dynamics of localised structures in that system clearly merits further research, ideally in the context of model classes, rather than using specific models.\\

Our numerical investigation with AUTO allows us to highlight similarities and differences with the analysis carried out in \cite{Carter2018}. Analogously to \cite{Carter2018}, the constructed orbits correspond to the upper branch of the $\mathcal{A}-\mathcal{C}$ bifurcation diagram shown in Figure \ref{fig:bifdiag}(a), whereas the lower branch is associated to solutions whose $U$-component remains asymptotically close to $U_\ast = \mathcal{A}$. When continuing in $\mathcal{H}$ up to the limit $\mathcal{H}=0$, AUTO correctly retrieves the single pulse solution constructed in \cite{Carter2018}: in this case, the toxicity ($S$) dynamics are still nontrivial, but they decouple from the vegetation ($V$) dynamics (see Figure \ref{fig:profiles}). It is worth noting that, while the numerics straightforwardly connect pulse solutions of \eqref{eq:prescaling} (for nontrivial $\mathcal{H}$) to the pulse solutions considered in \cite{Carter2018} (for $\mathcal{H}=0$), the geometric construction behind these two solution families is substantially different. One major difference is that the variable $w$ as introduced in \eqref{eq:definition_w_variable} contains an $s$-component which does not appear in \cite{Carter2018}, cf.~Figure \ref{fig:orbits}. Another crucial observation is that the equilibria $p_2$ and $p_3$ lose their hyperbolicity for $\mathcal{H}=0$, see \eqref{eq:eigenvalues_p2} and \eqref{eq:eigenvalues_p3}. In particular, this means that the arguments employed in the proof of Theorem \ref{thm:pertsol} cannot be used anymore, as these rely on the (normal) hyperbolicity of several geometric objects such as $\mathcal{M}^{(3)}$.\\
We can conclude that, from a purely mathematical perspective, the inclusion of autotoxicity desingularlises the pulse construction problem. Where in \cite{Carter2018}, the use of singular blowup techniques was needed to study the dynamics near the origin, the presence of the autotoxicity component $s$ suffices to drive the dynamics away from the origin, see Figure \ref{fig:proof2D}. In fact, the minimum distance of the slowest orbit segment to the origin is $\mathcal{O}(\delta \log \delta)$, cf.~\eqref{eq:w(s)_explicit_sol}. 
\\

We did not address the stability of the travelling pulse as a solution to the full PDE \eqref{eq:ndimmod}, which provides an interesting topic of future research. Furthermore, the extension of \eqref{eq:ndimmod} to include inertial effects \cite{Grifo2023} would provide a --to our knowledge-- hitherto unexplored analytical case study for the dynamics of far-from-equilibrium patterns in hyperbolic reaction-transport models.

\paragraph{Acknowledgements.}
AI and GG are members of Gruppo Nazionale per la Fisica Matematica (GNFM), Istituto Nazionale di Alta Matematica (INdAM), which has partially supported this work. AI acknowledges support from a scholarship (\emph{borsa per l'estero}) granted by INdAM to carry out a research stay at the University of Leiden and an FWF Hertha Firnberg Research Fellowship (T 1199-N). GG acknowledges support from MUR (Italian Ministry of University and Research) through PRIN2022-PNRR Project No.~P2022WC2ZZ ``A multidisciplinary approach to evaluate ecosystems resilience under climate change''.

\bibliographystyle{unsrtnat}
\bibliography{bibliography}

\end{document}